\documentclass[a4paper,english,reqno,10pt]{amsart}
\usepackage[T1]{fontenc}
\usepackage[utf8]{inputenc}
\usepackage{lmodern}
\usepackage{amsmath} 
\usepackage{amsthm} 
\usepackage{amsfonts}
\usepackage{amscd} 
\usepackage{amssymb} 
\usepackage{enumerate}
\usepackage{mathrsfs}
\usepackage[all]{xy}
\usepackage{mathtools}
\usepackage{stmaryrd} 
\usepackage{enumitem}
\usepackage{url}
\usepackage{microtype} 

\makeatletter
\@addtoreset{equation}{section}

\makeatother

\theoremstyle{plain}
\newtheorem{thm}[subsection]{Theorem}
\newtheorem{proposition}[subsection]{Proposition}
\newtheorem{lemma}[subsection]{Lemma}
\newtheorem{corollary}[subsection]{Corollary}
\newtheorem{ques}[subsection]{Question}

\theoremstyle{definition}

\theoremstyle{remark}
\newtheorem{remark}[subsection]{Remark}
\newtheorem{notation}[subsection]{Notation}

\DeclareMathOperator{\nSm}{nSm}
\DeclareMathOperator{\ord}{ord}
\DeclareMathOperator{\HS}{HS}

\newcommand{\Spec}{\mathop{\mathrm{Spec}}}
\newcommand{\Spf}{\mathop{\mathsf{Spf}}}
\newcommand{\Sch}{\mathop{\mathrm{Sch}}}
\newcommand{\Alg}{\mathop{\mathrm{Alg}}}
\newcommand{\arc}{\mathscr{L}_{\infty}}

\newcommand{\SingLoc}[1]{\nSm(#1)}
\newcommand{\Aff}{\mathbb{A}}
\newcommand{\Hom}{\mathrm{Hom}}
\newcommand{\Id}{\mathrm{Id}}
\newcommand{\Frac}{\mathrm{Frac}}
\newcommand{\Ker}{\mathrm{Ker}}

\newcommand{\acc}[2]{\left\langle #1\, ,\,#2 \right\rangle} 
\newcommand{\inv}{\times} 
\newcommand{\sumu}[1]{\underset{#1}{\sum}}

\newcommand{\sumsubu}[1]{\sumu{\substack{#1}}}

\newcommand{\eg}{{\it e.g.}\ }
\newcommand{\ie}{{\it i.e.}\ }
\newcommand{\opcit}{{\it op.cit.}}
\newcommand{\wt}{\widetilde}
\newcommand{\eps}{\varepsilon}
\let\phi\varphi

\newcommand{\ul}{\boldsymbol}

\newcommand{\dblbr}[1]{\llbracket #1 \rrbracket}

\newcommand{\vart}{t}

\newcommand{\numberset}[1]{\mathbb{#1}}

\newcommand{\torus}{\mathcal{T}}

\newcommand{\Lacp}{\mathfrak{LcCpl}}

\newcommand{\mds}{\mathcal{N}}

\newcommand\compl[1]{\widehat{#1}}

\newcommand{\diso}[1]{\{#1\neq 0\}}

\newcommand{\ann}[1]{\{#1=0\}}

\newcommand{\Fdis}[1]{F_{\bell_#1}} 
\newcommand{\Fdisu}[2]{F_{\bell_#1,#2}}
\newcommand{\struc}[1]{\mathcal{O}(#1)}

\newcommand{\fkh}{\mathfrak{h}}
\newcommand{\fkg}{\mathfrak{g}}
\newcommand{\fka}{\mathfrak{a}}
\newcommand{\fkb}{\mathfrak{b}}
\newcommand{\fkc}{\mathfrak{c}}
\newcommand{\fki}{\mathfrak{i}}
\newcommand{\fkj}{\mathfrak{j}}

\newcommand{\fkm}{\mathfrak{m}}
\newcommand{\fkM}{\mathfrak{M}}
\newcommand{\fkX}{\mathfrak{X}}

\newcommand{\cY}{\mathcal Y}

\newcommand{\cW}{\mathcal W}
\newcommand{\cZ}{\mathcal Z}

\newcommand{\cC}{\mathcal C}
\newcommand{\cA}{\mathcal A}
\newcommand{\cB}{\mathcal B}

\newcommand{\uX}{\ul{X}}
\newcommand{\uY}{\ul{Y}}
\newcommand{\uy}{\ul{y}}
\newcommand{\uJ}{\ul{J}}
\newcommand{\uK}{\ul{K}}
\newcommand{\uF}{\ul{F}}
\newcommand{\uH}{\ul{H}}
\newcommand{\uz}{\ul{z}}
\newcommand{\uw}{\ul{w}}
\newcommand{\uZ}{\ul{Z}}

\newcommand{\bn}{\ul{n}}
\newcommand{\br}{\ul{r}}
\newcommand{\bm}{\ul{m}}
\newcommand{\bell}{\ul{\ell}}
\newcommand{\be}{\ul{e}}

\newcommand{\cuY}{\ul{\cY}}

\newcommand{\oX}[1]{\ide{\uX}}
\newcommand{\oXY}[1]{\ide{\uX}}
\newcommand{\dbT}{\dblbr{\vart}}
\newcommand{\brT}{[\vart]}

\newcommand{\AX}{A[\uX]}

\newcommand{\AXY}{A[\uX,\uY]}

\newcommand{\KdX}{K\dblbr{\uX}}

\newcommand{\ide}[1]{\langle #1 \rangle}

\newcommand{\N}{\numberset{N}}

\newcommand{\Z}{\numberset{Z}}
\newcommand{\R}{\numberset{R}}

\title[Formal neighborhoods in arc scheme of toric singularities]{On the
behavior of formal neighborhoods in the Nash sets associated with
toric valuations: a comparison theorem}

\author{David Bourqui}
\address{Institut de recherche math\'ematique de Rennes\\ UMR 6625 du CNRS\\ Universit\'e de Rennes 1\\
Campus de Beaulieu\\
35042 Rennes cedex (France)}
\email{david.bourqui@univ-rennes1.fr}

\author{Mario Morán Cañón}
\address{Department of Mathematics\\ University of Oklahoma\\
601 Elm Avenue, Room 423\\
Norman, OK 73019 (USA)}
\email{mariomc@ou.edu}

\author{Julien Sebag}
\address{Institut de recherche math\'ematique de Rennes\\ UMR 6625 du CNRS\\ Universit\'e de Rennes 1\\
Campus de Beaulieu\\
35042 Rennes cedex (France)}
\email{julien.sebag@univ-rennes1.fr}

\subjclass[2010]{14B20, 14E18, 14M25}

\keywords{Arc schemes, formal neighborhoods, toric varieties}

\begin{document}

\maketitle

\begin{abstract}
We show that there exists a strong connection
between the generic formal neighborhood at a rational arc lying in the
Nash set associated with a toric divisorial valuation on a toric
variety and the formal neighborhood at the generic point of the same
Nash set. This may be interpreted as the fact that analytically along such a Nash set
the arc scheme of a toric variety is a product of a finite dimensional
singularity and an infinite dimensional affine space. 
\end{abstract}

\section{Introduction}\label{sec:introduction}

\subsection{}
This work establishes, in the case of toric singularities, a
comparison result between two kinds of formal neighborhoods in
arc schemes, which until now had been studied independently, in
particular in terms of motivations and involved techniques.

\subsection{}
The first class of formal neighborhoods we shall consider
are those of some finite-codimensional points of the arc scheme known
as stable points. Their study was motivated by the infamous Nash
problem (\cite{Nash}), which, loosely speaking, may be understood
as the problem of describing the connection between the resolution of
the singularities of a variety and the geometry of the arc space
associated with the variety.
An approach to this problem for surfaces was proposed
by Lejeune-Jalabert in terms of a problem of lifting of wedges (\cite{Lej:arcsan}).
Then Reguera extended the approach in higher dimensions,
putting forward the relevance of the study of the formal neighborhoods of the generic points of the
so-called Nash sets associated with divisorial valuations, which are
the prototypical examples of stable points (\cite{Reg:CSL,Reguera}).
She showed in particular that these formal neighborhoods are noetherian
(later in \cite{deFernex-Docampo}, a new proof of the result, as well as a
proof of the converse statement, were given), allowing her to establish a curve selection lemma for
arc spaces which was crucially used in subsequent works on the Nash problem
(\cite{MR2979864,deFDoc:terminal,MR2905305}).
She pointed out that there should exist a strong connection between
the algebraic properties of the formal neighborhood of the generic
point of the Nash set associated with a valuation and the geometric
properties and invariants of the valuation (see \eg \cite[Corollary 5.12]{Reguera}).

Later in \cite{MouReg} Mourtada and Reguera showed that the embedding
dimension of such a formal neighborhood equals the Mather discrepancy of
the valuation ; they also pointed out the interest and the difficulty of
computing the dimension of these formal neighborhoods.

\subsection{}
The second kind of formal neighborhoods we are interested in are those of
non-degenerate {\em rational} arcs (here non-degenerate means not entirely contained in the singular locus of the
variety).  Contrarily to the aforementioned case of stable points, these
neighborhoods are not noetherian in general. Yet Grinberg and Kazdhan (when the
base field is the field of complex numbers, \cite{GK}) and then
Drinfeld (over any field, \cite{DGKv1}) established a striking
finiteness property in this context, showing that any formal
neighborhood of this kind may be written as the product of an infinite
smooth factor (\ie countably many copies of the formal disk) with the
formal neighborhood of a rational point of a scheme of finite type (a
"finite formal model" of the arc under consideration).  The work of
Drinfeld, Grinberg and Kazdhan was motivated by geometric
representation theory and Langlands program; subsequent works in this
direction include \cite{BouNgoSak,BouCohArc,Ngo:Wei}.  On the other
hand, the first and last named authors of the present paper suggested
in \cite{BouSeb:confluentes} that the finite formal models of a
rational arc could be interpreted in terms of the singularity at the
origin of the arc.  In \cite{BouSeb:smooth:arcs} they showed that the
origin of the arc is smooth if and only if it admits a trivial finite
formal model, supporting the fact that the finite formal models should
provide a sensible measure of the complexity of the involved
singularities. In \cite{BS-IJM} the case of monomial plane curves
singularities is explored whereas in \cite{Bou-Seb:toric} finite
formal models of rational arcs on toric varieties are studied.
Here it is also worth mentioning the recent work
\cite{ChDeFDoc:embcodim} of Chiu, de Fernex and Docampo showing in
particular that the non-degeneracy condition in Drinfeld-Grinberg-Kazdhan theorem is
also necessary (previous related works are \cite{BouSeb:false} and \cite{ChHa:degen}.

\subsection{} Summing up, the two classes of formal neighborhoods (non-degenerate
stable points / rational arcs) on the arc space of an algebraic
variety share two common features:
a finiteness property and interesting (and potential in full generality) connections
with the singularities of the variety. Yet for each class the
manifestation of these properties seems to take rather different forms.
It is natural to ask whether the phenomena observed for both classes
are the emanation of a common mechanism.

The aforementioned work \cite{Bou-Seb:toric} provides the first piece
of evidence that there may exist a strong connection between the two
classes of formal neighborhoods, at least in the case of toric
singularities. On one hand, the formula obtained there for
the embedding dimension of the minimal formal model of a
non-degenerate arc on a toric variety (Theorem 1.4 of \opcit)
presents a striking similarity with the one obtained more generally by Reguera and Mourtada in the
context of stable points (\cite[Theorem 3.4]{MouReg}).
On the other hand, the fact that the formal neighborhood of a generic
rational arc of the Nash set associated with a toric valuation is
constant (\cite[Theorem 1.3]{Bou-Seb:toric}) makes it senseful to ask
whether such a formal neighborhood can be compared with the formal
neighborhood of the generic point of the Nash set; the question was 
explicitly raised in \cite[Question 7.13]{Bou-Seb:Nash}.

\subsection{}
In the present work, we answer positively the question, thus establishing for the first time 
a strong direct connection between the two important classes of formal
neighborhood discussed before. More precisely, we show the following result.

\begin{thm}
\label{thm:main2}
Let $k$ be a field of characteristic zero. Let $V$ be an affine
normal toric $k$-variety. Let $v$ be a toric divisorial valuation on
$V$ and $\mds_v$ be the associated Nash set.
Let $\eta_v$ be the generic point of $\mds_v$ and $\kappa_v$ be the
residue field of $\eta_v$.

For a general $k$-rational arc $\alpha\in \mds_v$ there exists an
isomorphism of $\kappa_v$-formal schemes between
$\compl{\arc(V)_{\eta_v}}\hat\otimes_{\kappa_v}\kappa_v\dblbr{(T_i)_{i\in\N}}$
and $\compl{\arc(V)_{\alpha}}\hat\otimes_k \kappa_v$.
\end{thm}

Geometrically, this might be interpreted as the fact that $\arc(V)$ is
analytically along $\mds_v$ a product of a finite dimensional
singularity and an infinite dimensional affine space.

\subsection{}
We stress that theorem \ref{thm:main2} is a direct consequence of
theorem \ref{thm:comparison} stated and proved below, which is our main statement and gives a more precise result. In
particular, it provides an explicit description of the formal
neighborhood of the generic point of the Nash
set associated with a toric valuation in terms of a noetherian
formal scheme associated with the same valuation, which
had been introduced in the previous work \cite{Bou-Seb:toric}. 
Some geometric consequences of this description are explained below. 

\subsection{}
We now describe the strategy and techniques involved in the proof of
theorem \ref{thm:comparison}. First we insist on the fact that
although the results of \cite{Bou-Seb:toric} inspired the present
work, there only the case of formal neighborhood of rational arcs was
treated. The main point of the present article is to show that the result
obtained there agrees, in some sense, with the one obtained for
generic points of toric Nash sets. It is important to note that
the techniques in \opcit\ can not be directly imported.
The reason is that one of the crucial ingredients in \opcit\ is the interpretation of
the formal neighborhood of a rational arc as a parameter space for the
infinitesimal deformations of the arc, allowing in particular to take
full advantage of the Weierstrass preparation and division theorems.
This point of view, already used by Drinfeld in \cite{DGKv1}, provides a very
efficient tool in the context of rational arcs, but to the best of our
knowledge no sensible analog exists in the context of stable points,
and this constitutes a major technical issue for the proof of the
comparison theorem.

The important steps in the proof of our result are the following.
\begin{enumerate}
\item As a preliminary step, we obtain  
  a sensible presentation of the formal neighborhood of the
  generic point of a toric Nash set; this is done in section
  \ref{sec:technical:formal:neighbourhood} in an ``abstract'' way and
  some extra work is needed in section \ref{sec:toric:varieties} in
  order to  apply the results to the toric context.
  This kind of computation was considered
  more generally by Reguera in \cite{Reguera,Reg:coo} and Reguera and Mourtada in \cite{MouReg}.
  However, the results as given there are not adapted to our purposes, 
  and our approach, based on a direct application of a version of the
  Hensel's lemma for an infinite set of variables, is somewhat
  different. We hope that this kind of approach may shed
  some new light on the computation of formal neighborhoods of stable points.
\item We compare the obtained presentation with to the one obtained in \cite{Bou-Seb:toric} for
rational arcs. The basic idea is that despite no interpretation in
terms of deformations of the presentation exists, we are still able to
transform it using the Weierstrass factorization and division theorems
in a rather intricate way, ending up with the conclusion that in some
sense it coincides with the finite model obtained in \opcit. The
technical core of the argument is in section
\ref{sec:technical:comparison}, the results of which are applied in
section \ref{sec:toric:varieties} to the toric context.
\end{enumerate}

\subsection{} As a direct by-product of our main result, we obtain 
a formula for the dimension and the number of irreducible components
of the formal neighborhood of the generic point of a Nash toric set
(see corollary \ref{coro:dim:irrcomp} for the precise statement).

This in turn implies the following statement.
Recall that a
  divisorial valuation $v$ on an algebraic variety $V$ is {\em essential} 
  if for every resolution $W\to V$ of the singularities of $V$, the center of $v$ on $W$ is an
    irreducible component of the exceptional locus of the resolution.
Here we say that such a divisorial valuation is {\em strongly essential} 
if for every resolution $W\to V$ of the singularities of $V$, the center of $v$ on $W$ is an
  irreducible component {\em of codimension $1$} of the exceptional locus of the resolution.

\begin{thm}\label{theo:strong:essential}
Let $V$ be a toric variety and $v$ a divisorial toric valuation on
$V$, centered in the singular locus of $V$. Let $\eta_v$ be the
generic point of the Nash set associated with $v$.
Then the formal neighborhood $\compl{\arc(V)_{\eta_v}}$
of $\eta_v$ in the arc scheme  $\arc(V)$ associated with $v$
is of dimension $1$ if and only if $v$ is strongly essential.
Moreover, in this case, $\compl{\arc(V)_{\eta_v}}$ is
irreducible and the associated reduced formal scheme is a formal disk.
\end{thm}
Recall that Reguera showed that in general, if $\eta_v$ is the
generic point of the Nash set associated with an essential divisorial valuation
$v$ on a an algebraic variety $V$, then $\compl{\arc(V)_{\eta_v}}$ is
irreducible of dimension $1$ if and only if $V$ satisfies a
property of lifting wedges centered at $\eta_v$, and that this
condition implies that $v$ is a Nash valuation (see \cite[Corollory 5.12]{Reguera})
Note that the latter property of lifting wedges is stronger than the one
considered in \cite[Section 5]{Reg:CSL}, which was shown to be
equivalent to the fact that $v$ is a Nash valuation.

In view of theorem \ref{theo:strong:essential} and of the results of \cite{Bou-Seb:toric},
the following question seems natural.
\begin{ques}\label{ques:irrdim1}
Let $\eta_v$ is the generic point of the Nash set associated with an essential divisorial valuation
$v$ on a an algebraic variety $V$. Are the following conditions equivalent?
\begin{enumerate}
\item 
The formal neighborhood $\compl{\arc(V)_{\eta_v}}$ is irreducible of
dimension $1$.
\item
The valuation $v$ is Nash and strongly essential.
\item
The minimal formal model of a generic rational arc the Nash set
associated with $v$ is irreducible of dimension $0$
\end{enumerate}
\end{ques}
Our results show that the answer is positive in case $v$ is a toric
valuation on a toric variety.
Note also that if $v$ is a terminal valuations (hence strongly essential),
$v$ is a Nash valuation and it is known that (1) holds (see \cite{deFDoc:terminal} and
\cite[Corollary 4.3]{MouReg}).

\subsection{}
Of course, it is natural to ask whether there exist other classes of varieties
for which the formal neighborhood of a rational non degenerate arc is
generically constant on Nash sets, and, if it is so, whether a
comparison result akin to theorem \ref{thm:comparison} still holds. In
\cite{Bou-Seb:Nash}, it is observed that for curve singularities the
genericity property holds. For this class of singularities, the first
and second authors of the present work will address the question of the
validity of the comparison theorem in a forthcoming paper. For normal
varieties equipped with a ``big'' action of an algebraic group
(typically for spherical varieties), it is very likely that at least
the genericity property holds. It would also be interesting to be
able to connect this kind of property to the status of the 
Nash problem (which is known to admit a positive answer in particular for curve and
toric singularities). 
\subsection{}
It is our hope that understanding more precisely the connection
between formal neighborhoods of rational and stable points on arc
spaces could open the way to the study of stable points via the
deformation-theoretic point of view which is so useful in the context
of rational points. At the end of the paper, we give an example of
a computation for a toric valuation illustrating how our comparison
theorem provides a more tractable result than a direct approach.

\section{General conventions and notation}\label{sec:conventions}

\subsection{}  Throughout the whole article, we designate by $ k $ a field of
characteristic zero.  $\Alg_k$ (resp. $\Sch_k$) is the category of
$k$-algebras (resp. of $k$-schemes).  If $ K $ is a field extension of $ k $, we denote by $\Lacp_K$ the
category of complete local $k$-algebras with residue field $k$-isomorphic to $K$. 
For any category $\cC$ and any objects $A,B\in \cC$, we denote by
$\Hom_{\cC}(A,B)$  the set of morphisms from $A$ to $B$ in the category.

\subsection{} A \emph{$k$-variety} is a $k$-scheme of finite type. The non-smooth locus of the structural
morphism of a $ k $-variety $ V $ is the \emph{singular locus of $V$} and its associated reduced $ k $-scheme is denoted by $  \SingLoc{V} $.
If $V$ is an affine $k$-variety and $f$ is a regular function on $V$,
we denote by $\diso{f}$ the distinguished open subset of $V$ where $f$
does not vanish and by $\ann{f}$ the closed set $V\setminus \diso{f}$.

\subsection{}\label{ssec:extension ideals}
Let $R$ be a ring, let $ \fki $ be an ideal of $ R $ and $f\in R$.
We denote by $R_f$ the localization of $R$ with respect to the
multiplicative subset $\{f^r;\,r\in\N\}$.
We denote by $\fki :f^{\infty}$ the ideal $\{g\in R:\,f^rg\in \fki \text{ for some } r\in\N \}$.
Let $R'$ be another ring and $ \vartheta:R\rightarrow R' $ a morphism
of rings. For the sake of easy reading and abusing notation, the
extension ideal of $ \fki $ in $R'$
via the morphism $\vartheta$ is denoted by $\vartheta(\fki )$, or even
by $\fki$ if the involved morphism $\vartheta$ is clear from the
context (for example if $R$ is a subring of $R'$).

\subsection{}
Let $ R[X_{\omega};\omega\in\Omega] $ be a polynomial ring and $ f\in
R $. Let $ S $ be a $ R $-algebra and $\{s_{\omega}\}_{\omega\in\Omega}$ a collection of elements in
$S$. Then we denote by $ f|_{X_{\omega}=s_{\omega}}\in S$ the image of
$f$ by the unique morphism of $R$-algebra $R[X_{\omega}]\to S$ mapping
$ X_{\omega} $ to $ s_{\omega} $ for each $ \omega\in\Omega $.

\subsection{}\label{subsec:WPT}  Let $(\cA,\fkM_{\cA})$ be a complete local ring. An
element $f=\sum_{i\in \N} f_i\vart^i\in \cA\dbT$ is {\em regular} if $f\notin \fkM_{\cA}\dbT$. 
Its {\em order} is $\inf\{i\in \N,\,\,f_i\notin \fkM_{\cA}\}$. Let
$d\in \N$. A {\em Weierstrass polynomial} of order $d$ 
is a monic polynomial of degree $d$, whose order as a regular element of $\cA\dbT$ is $d$.
We shall make a crucial use of the following classical results (the
Weierstrass division and  preparation theorems, see \eg \cite[Theorems 9.1 \& 9.2]{Lang}.
Let $f\in \cA\dbT$ be a regular element of order $d$. Then:
\begin{itemize}\renewcommand{\labelitemi}{$\bullet$}
\item[(i)] There exists a unique pair $(p(\vart),u(\vart))\in \cA\dbT^2$ such that 
$f(\vart)=p(\vart)u(\vart)$, $p(\vart)$ is a Weierstrass polynomial of degree $d$ and
  $u(\vart)$ is a unit in $\cA\dbT$.
\item[(ii)] Let $g\in \cA\dbT$; then there exists a unique pair $(q(\vart),r(\vart))\in \cA\dbT^2$
  such that $g(\vart)=f(\vart)q(\vart)+r(\vart)$ and $r(\vart)$ is a polynomial of degree $<d$.
\end{itemize}
Note that in particular any regular element of $\cA\dbT$ is not a zero
divisor in $\cA\dbT$.

\section{Recollection on arc scheme and toric varieties}\label{sec:recoll:arc:toric}
The crucial objects of our study are arc schemes and toric
varieties. For the convenience of the reader, we give in this section
an overview of the main definitions and properties that we will use in
the article. Along the way, we fix some notation and state and
prove some technical lemmas useful for the sequel.

\subsection{}\label{ssec:jetarc}
Since we are only interested in local properties of arc schemes, we
limit ourselves to the case of arc schemes associated with affine
varieties. Proofs as well as more details on the general theory of
arc schemes are to be found \eg in \cite{CNS:book,BNS-Proceedings}.

To every affine $k$-variety $V$ one attaches its \emph{arc scheme}
$\arc(V)$ which is an affine $k$-scheme characterized  by the
fact that for every $k$-algebra $R$ one has a functorial bijection
\begin{equation}\label{equa:funpoints:arcscheme}
\Hom_{\Sch_k}(\Spec(R),\arc(V))\cong \Hom_{\Sch_k}(\Spec(R\dbT),V).
\end{equation}
A point of $\arc(V)$ is called an \emph{arc}. The above functorial bijection
and the $k$-algebra morphism $R\dbT\to R$ mapping $t$ to $0$ induces a
morphism of $V$-scheme $\arc(V)\rightarrow V$ 
which sends an arc $\alpha$ to its \emph{base-point} $\alpha(0)$.

We will need explicit equations of the affine scheme $\arc(V)$ in
terms of equations of the affine $k$-variety $V$. We begin with the
case of the affine space. Let $\uZ=\{Z_1,\dots,Z_h\}$ be a finite set of indeterminates.
Consider the ring $k[\uZ_{\infty}]:=k[Z_{i,j}\,:\,i\in\{1,\dots,h\},j\in\N]$,
and the $k$-algebra morphism $\phi\colon k[\uZ]\to k[\uZ_{\infty}]$ mapping
$Z_i$ to $Z_{i,0}$. Then the affine $k$-scheme $\arc(\Aff^h_k)$ is isomorphic to
$\Spec(k[\uZ_{\infty}])$. The morphism $\phi\colon k[\uZ]\to k[\uZ_{\infty}]$
induces a morphism $\struc{\Aff^h_k}\to \struc{\arc(\Aff^h_k)}$ dual to the morphism $\alpha\mapsto \alpha(0)$.

For every $F\in k[\uZ]$, define $\{F_s\}_{s\in \N}\in k[\uZ_{\infty}]^{\N}$
by the following identity in $k[\uZ_{\infty}]\dbT$:
\begin{equation}\label{equa:HS}
F|_{Z_i=\sum\limits_{s\in\N}Z_{i,s}\vart^{s}}=\sum\limits_{s\in\N} F_{s}\vart^s.
\end{equation}
Note that $\HS\colon F\mapsto \sum\limits_{s\in\N} F_{s}\vart^s$ is a morphism
of $k$-algebras $k[\uZ]\to k[\uZ_{\infty}]\dbT$.
If $\fki$ is an ideal of $k[\uZ]$, generated by a family $\{F_{\delta}:\delta\in \Delta\}$,
the ideal $\ide{F_{\delta,s}\,:\,\delta\in \Delta,\,s\in \N}$ does not
depend on the choice of the generating family. We denote it by
$[\fki]$. (This notation is borrowed from differential algebra; for
more information on the link between differential algebra and arc
schemes, see \eg \cite{BNS-Proceedings}.)
The following lemma will be useful.
\begin{lemma}\label{lemm:loc:fra}
Let $\fki$ be an ideal of $k[\uZ]$. Let $d\leq h$ and $F\in k[\uZ]$
such that $F$ lies in the quotient ideal $\fki:(\prod_{i=1}^dZ_i)^{\infty}$. Let $(k_i)\in \N^d$ and $\fka$
be the ideal $\ide{Z_{i,s}\,:\,1\leq i\leq h,\,0\leq s\leq k_i}$ of $k[\uZ_{\infty}]$.
Let $G=\prod_{i=1}^dZ_{i,k_i}$.
Then in the localization $k[\uZ_{\infty}]_{G}$, the ideal $[\ide{F}]$
is contained in the ideal $[\fki]+\fka$.
\end{lemma}
\begin{proof}
Let $H\in \fki$ and $N\in \N$ such that
$(\prod_{i=1}^dZ_i)^NF=H$. Applying $\HS$ and using the very
definition of $\fka$, one obtains the relation
\[
\prod_{i=1}^d(\vart^{k_i}[Z_{i,k_i}+\vart(\dots)])^N\HS(F)=\HS(H)\pmod{\fka\dbT}.
\]
Thus, setting $K:=\sum_{i=1}^d Nk_i$, for $s<K$ one has $H_s\in \fka\dbT$ and 
one may write
\[
\prod_{i=1}^d[Z_{i,k_i}+\vart(\dots)]^N\HS(F)=\sum_{s\geq 0}H_{s+K}\vart^{s}\pmod{\fka\dbT}.
\]
By the definition of $G$, the series $\prod_{i=1}^d[Z_{i,k_i}+\vart(\dots)]^N$ is invertible in $k[\uZ_{\infty}]_{G}\dbT$.
That concludes the proof.
\end{proof}
Now if $\fki$ is an ideal of $k[\uZ]$ and the affine $k$-scheme $V$
is presented as $\Spec(k[\uZ]/\fki)$, then the affine $k$-scheme $\arc(V)$ is isomorphic to
$\Spec(k[\uZ_{\infty}]/[\fki])$. The morphism $\phi\colon k[\uZ]\to k[\uZ_{\infty}]$
induces a morphism $\struc{V}\to \struc{\arc(V)}$ dual to the morphism $\alpha\mapsto \alpha(0)$.

The morphism $\HS$ induces a morphism
$\HS_V\colon \struc{V}\to \struc{\arc(V)}\dbT$, dual to the so-called
{\em universal arc on $V$}. If $R$ is a $k$-algebra and $\alpha^{\ast}\colon \struc{\arc(V)}\to R$ is an
$R$-point of $\arc(V)$, inducing a $k$-algebra morphism
$\alpha^{\ast}\dbT\colon \struc{\arc(V)}\dbT\to R\dbT$,
the corresponding $R\dbT$ point of $\struc{V}$ by bijection \eqref{equa:funpoints:arcscheme}
is $\alpha^{\ast}\dbT\circ \HS_V$.

Let $W$ be a closed $k$-subscheme of $V$ and $\fkj=\ide{G_{\gamma}\,:\,\gamma\in\Gamma}$ be and
ideal of $k[\uZ]$ such that $W\cong\Spec(k[\uZ]/\fki+\fkj)$.
Then
\[
\arc(W)\cong \Spec(k[\uZ_{\infty}]/[\fki]+[\fkj])
\]
identifies with a closed subscheme of $\arc(V)$ and the open subset
$\arc(V)\setminus \arc(W)$ of $\arc(V)$ is the union of the
distinguished open subsets
$\diso{G_{\gamma,s}}$ for $\gamma\in \Gamma$ and $s\in \N$.

An element of  $\arc(V)\setminus \arc(\SingLoc{V}) $ is called a \emph{non-degenerate arc}.

\subsection{}\label{ssec:Nash set}
  (See, e.g., \cite{ELL04,Ish:arcs:valuations,Ish:maximal}.)
  Let $\alpha$ be an arc of $V$ with residue field $\kappa(\alpha)$,
  inducing a $\kappa(\alpha)\dbT$-point of $V$. Composing the morphism $\struc{V}\to
  \kappa(\alpha)\dbT$ with the $t$-valuation defines a semivaluation
  $\ord_{\alpha}\colon \struc{V}\to \N\cup\{+\infty\}$. Now let $v$ be a
  divisorial valuation over $V$. The associated {\em Nash set}, or {\em
    maximal divisorial set}, is the closure in $\arc(V)$ of the set
  $\{\alpha\in \arc(V),\quad \ord_{\alpha}=v\}$. It is an irreducible
  subset of $\arc(V)$, denoted by $\mds_v$.

\subsection{}
From now on, we introduce some notation and basic facts on normal
toric varieties. (For further details, e.g., see \cite[Sections 1.1
  and 1.2]{CLS}.) Since we are studying local
properties, in this article we can restrict ourselves to the case of affine normal toric
varieties.

Let $ d $ be a positive integer and $ \torus $
a split algebraic $ k $-torus of dimension $ d $. Let $
N:=\Hom(\numberset{G}_{m,k},\torus) $ be the group of its
cocharacters which is a free $ \Z $-module of rank $ d $ (i.e., a
lattice isomorphic to $ \Z^d $) and $ M:=\Hom_{\Z}(N,\Z) $ its dual $
\Z $-module (i.e., the group of characters of $ \torus $). 
 We denote by $ N_{\R}=N\otimes_{\Z}\R $
(resp. $ M_{\R}=M\otimes_{\Z}\R $) the $ \R $-vector space of
dimension $ d $ associated with $ N $ (resp. $ M $). We have a $ \R
$-bilinear canonical map $ \ide{\; ,\;} : M_{\R}\times
N_{\R}\rightarrow \R $. The points of the lattices $N$ and $M$, considered as
points of the associated vector spaces, are called their integral
points. We will simply call a \emph{cone} of $ N_{\R} $ a strongly
convex rational polyhedral cone of the vector space $ N_{\R} $ (\ie a
convex cone generated by finitely many elements of $ N $, which
moreover does not contain any line).

\subsection{}\label{ssec:toric:basic}
Let $ \sigma $ be a cone of $ N_{\R} $.  
By the Gordan Lemma (e.g., \cite[Proposition 1.2.17]{CLS}),
the semigroup $S_{\sigma}:=\sigma^{\vee}\cap M $ is finitely
generated. The spectrum of the $k$-algebra $k[S_{\sigma}]$ associated with the semigroup $S_{\sigma}$
then defines a normal affine toric variety $V_{\sigma}$ with torus
$\torus$. Note that every affine normal toric variety with torus
$ \torus $ is of the form $ V_{\sigma} $ for $ \sigma $ a cone of
$ N $ (see e.g., \cite[Theorem 1.3.5]{CLS}).
For every $ \bm\in S_{\sigma} $ we denote by $ \chi^{\bm} $ the regular
function on $ V_{\sigma} $ defined by $ \bm $. Recall that for every
$k$-algebra $A$, the set $\Hom_{\Alg_k}(k[S_{\sigma}],A)$ is in
natural bijection with the set of semigroup morphism $S_{\sigma}\to A$, where the semigroup structure on $A$ is induced by the multiplication.

 Let  $\{\bm_{1},\dots,\bm_{h}\}$ be the minimal set of generators of the
 semigroup $S_{\sigma}$. We may and shall assume in the sequel that
 the set $\{ \bm_{1},\dots,\bm_{d}\}$
is a $\Z$-basis of $M$.
 If we call $ z_1,\dots,z_h $ respectively
 $\chi^{\bm_1},\dots,\chi^{\bm_h} $, we deduce that
 $\struc{V_{\sigma}}:=k[S_{\sigma}]= k[z_1,\dots,z_h]$. Moreover the closed
subscheme defined by the ideal $\ide{\prod_{1\le i\le h}z_i}$ has for
support the closed set $V_{\sigma}\setminus \torus$, and the same holds for the
ideal $\ide{\prod_{1\le i\le d}z_i}$.

Now let $\arc^{\circ}(V_{\sigma})$ be the open subset of  $\arc(V_{\sigma})$ defined by 
$
\arc^{\circ}(V_{\sigma}):=\arc(V_{\sigma})\setminus \arc(V_{\sigma}\setminus \torus).
$
Thus, by \S\ref{ssec:jetarc}, for any $\alpha\in\arc(V_{\sigma})$,
one has $\alpha\in \arc^{\circ}(V_{\sigma})$ if and only if for every 
$\bm\in S_{\sigma}$, $\alpha^{\ast}(\chi^{\bm})\neq 0$ if and only if 
for every $1 \le i \le d$ one has $\alpha^{\ast}(z_i)\in \kappa(\alpha)\dbT\setminus\{0\}$. 
Therefore one has
\[
\arc^{\circ}(V_{\sigma})
=\bigcap_{i=1}^d \bigcup_{s\in \N}\diso{z_{i,s}}.
\]

\subsection{}\label{ssec:toric:valuation}
Let $\bn$ be a point of $ \sigma\cap N $.  For $f\in \struc{V_{\sigma}}$,
set
\[
\ord_{\bn }(f)=\acc{\bn}{f}:=\inf\limits_{\chi^{\bm}\in f} \acc{\bm}{\bn}
\]
(here $\chi^{\bm}\in f$ means that the coefficient of $\chi^{\bm}$ in the
decomposition of $f$ is not zero). Then $\ord_{\bn}$ is a divisorial toric valuation on $V_{\sigma}$,
and one easily sees that $\bn\mapsto \ord_{\bn}$ is a bijection
between $\sigma\cap N$ and the set of divisorial toric valuations on
$V_{\sigma}$. From now on we shall identify the later set with $N\cap \sigma$.

Let $\alpha\in \arc^{\circ}(V_{\sigma})$. The semigroup morphism
$\bm\in S_{\sigma}\mapsto \ord_{\vart}(\alpha^{\ast}(\chi^{\bm}))$
extends uniquely to a group morphism $\bn_{\alpha}\colon M\to \Z$ which is nonnegative on
$S_{\sigma}$.
In other words $\bn_{\alpha}$ is the unique element of $N\cap \sigma$
satisfying: for every $1\le i \le h$,
$\ord_{\vart}(\alpha^{\ast}(z_i))=\acc{\bm_i}{\bn_{\alpha}}$.
For every $ \bn \in \sigma\cap N $, we set 
\[
\arc^{\circ}(V_{\sigma})_{\bn }:=\{\alpha\in \arc^{\circ}(V_{\sigma})\,;\,\bn_{\alpha}=\bn \}
\]
\[
\text{and}\quad\arc^{\circ}(V_{\sigma})_{\geq \bn }:=\{\alpha\in   \arc^{\circ}(V_{\sigma})\,;\,\bn_{\alpha}\in \bn+\sigma \}.
\]
Thus $\alpha\in \arc^{\circ}(V_{\sigma})_{\geq \bn}$ if and
only if for every $\bm\in S_{\sigma}$ one has
$\ord_{\vart}(\alpha^{\ast}(\chi^{\bm}))\geq \acc{\bm}{\bn}$ if and only if 
for every $1\le i \le h$ one has
$\ord_{\vart}(\alpha^{\ast}(\chi^{\bm_i}))\geq \acc{\bm_i}{\bn}$.
If $\alpha\in \arc^{\circ}(V_{\sigma})_{\geq \bn}$ and
$\phi_{\alpha}\colon S_{\sigma}\to \kappa(\alpha)\dbT$ is the
associated semigroup morphism, then $\bm\mapsto \vart^{-\acc{\bm}{\bn}}\phi_{\alpha}$
defines a semigroup morphism $\psi_{\alpha}\colon S_{\sigma}\to \kappa(\alpha)\dbT$,
and $\bn_{\alpha}=\bn$ if and only if
$\psi_{\alpha}(S_{\sigma})\subset (\kappa(\alpha)\dbT)^{\inv}$ if and
only if for $1\le i\le d$ one has $\ord_{\vart}(\alpha^{\ast}(\chi^{\bm_i}))=\acc{\bm_i}{\bn}$.
Note also that the element of $\arc(V_{\sigma})(k)$ corresponding to the semigroup morphism
$S_{\alpha}\to k\dbT$, $\bm\mapsto \vart^{\acc{\bm}{\bn}}$ lies in $\arc^{\circ}(V_{\sigma})_{\bn}$, which is therefore nonempty.

The following lemma will be useful for describing the generic points
of the Nash sets associated with toric valuations.
\begin{lemma}\label{lemm:decr:mds:toric:set}
Let $\bn \in \sigma\cap N $.
\begin{itemize}
\item [(i)]
One has
\[
\arc^{\circ}(V_{\sigma})_{\geq \bn}=\arc^{\circ}(V_{\sigma})\cap
\bigcap_{i=1}^h \bigcap_{s=0}^{\acc{\bm_i}{\bn}-1} \ann{z_{i,s}}.\]
\item [(ii)]
One has
\[
\arc^{\circ}(V_{\sigma})_{\bn}=\arc^{\circ}(V_{\sigma})\cap
\left(\bigcap_{i=1}^h    \bigcap_{s=0}^{\acc{\bm_i}{\bn}-1} 
\ann{z_{i,s}}\right)
\cap
\bigcap_{i=1}^d\diso{z_{i,\acc{\bm_i}{\bn}}}.
\]
\item [(iii)]
The closure of $\arc^{\circ}(V_{\sigma})_{\bn }$ coincides with the
Nash set $\mds_{\bn}=\mds_{\ord_{\bn}}$ (see \S \ref{ssec:Nash
set}) associated with the toric valuation $\bn$.
\item [(iv)]
One has
\[
\arc^{\circ}(V_{\sigma})_{\bn}=\mds_{\bn}\cap \arc^{\circ}(V_{\sigma}) \cap \bigcap_{i=1}^d\diso{z_{i,\acc{\bm_i}{\bn}}}.
\]
In particular $\arc^{\circ}(V_{\sigma})_{\bn}$ is a nonempty open subset of $\mds_{\bn}$.
\end{itemize}
\end{lemma}
\begin{proof}
Assertion (i) and (ii) are nothing but a reformulation of the above descriptions
of $\arc^{\circ}(V_{\sigma})_{\bn}$ and $\arc^{\circ}(V_{\sigma})_{\geq \bn}$.

For a proof of (iii), see \cite[Example 2.10]{Ish:maximal}.

Assertion (iv) is a straightforward topological consequence of (ii) and (iii).
\end{proof}

\subsection{}\label{subsec:toric:ideal}
An explicit description of $ V_{\sigma} $ as a closed subscheme of the
affine space $\Aff_{k}^{h} $ will be useful in the sequel.

Recall that $\{\bm_1,\dots,\bm_h\}$ is the minimal set of generators of
$S_{\sigma}$, and that we may and shall assume that $\{\bm_1,\dots,\bm_d\}$
is a $ \Z$-basis of $ M $.
Let $\{\be_i;\,i\in \{1,\dots,h\}\}$ be the canonical basis of $\Z^h$.
Being given $ \bell=(\ell_{1},\dots,\ell_{h})\in \Z^h$, we set
\[
\bell^{+}=\sum\limits_{\ell_{i}\ge 0}\ell_{i}\be_i
\;\;\;\text{and}\;\;\; \bell^{-}=-\sum\limits_{\ell_{i}<
  0}\ell_{i}\be_i,
\]
which are both elements of $\N^{h}$. Note that $\bell=\bell^{+}-\bell^{-} $.
On the other hand, for $\bell\in \N^h$, set $\uZ^{\bell}:=\prod\limits_{i=1}^{h}Z_{i}^{\ell_{i}}$
and $F_{\bell}:=\uZ^{\bell^{+}}-\uZ^{\bell^{-}}$.

Mapping $\be_i$ to $\pi(\be_i):=\bm_i$
induces an exact sequence of groups
\begin{align}\label{eq:suite exacte groupes}
0\rightarrow L \rightarrow \Z^{h}\xrightarrow{\pi}M\rightarrow 0,
\end{align}
where $ L $ is a subgroup of $ \Z^{h} $. For $\bell\in L$ and $\bn\in N$, we set 
\[
\acc{\bn}{\bell}:=\acc{\bn}{\sum\limits_{\substack{i=1\\\ell_i>0}}^h\ell_i\bm_i}
=-\acc{\bn}{\sum\limits_{\substack{i=1\\\ell_i<0}}^h\ell_i\bm_i}.
\]
By \cite[Proposition 1.1.9]{CLS}, the ideal of $ k[\uZ] $ defining $V_{\sigma} $ is
\begin{equation}\label{eq:toric ideal}
\fki_{\sigma}:=\ide{F_{\bell}\,;\,\bell\in L }.
\end{equation}

The set $\{\bm_{1},\dots,\bm_{d}\} $ being a $ \Z $-basis of $M$, 
for $ q\in \{d+1,\dots,h\}$,   we can write the element $\bm_q $ as a linear
combination with (possibly negative) integer coefficients  of
$\bm_{1},\dots,\bm_{d} $.
Thus we have in $ L $ an element $\bell_q=(\ell_{q,1},\dots,\ell_{q,h}) $ such that $ \ell_{q,q}=1 $
and $ \ell_{q,q'}=0 $ for every $ q'\in\{d+1,\dots,h\}\setminus\{q\}$. The element $ \bell_q\in L $ induces an element
\begin{align}\label{eq:F_q}
\Fdis{q}=Z_q\prod\limits_{\substack{i=1\\\ell_{q,i}\ge 0}}^{d}Z_{i}^{\ell_{q,i}}-\prod\limits_{\substack{i=1\\\ell_{q,i}< 0}}^{d}Z_{i}^{-\ell_{q,i}}
\end{align}
\noindent
in the ideal $ \fki_{\sigma} $. We observe that in the binomial
$\Fdis{q} $ none of the variables $ Z_{d+1},\dots,Z_{h} $ appears,
excepting $ Z_q $.
\begin{lemma}\label{lemm:quotient}
Let $\fkj:=\ide{\Fdis{q}\,:\,d+1\leq q\leq h}$. 
\begin{enumerate}
\item [(i)] Set $G_d:=\prod\limits_{i=1}^dZ_i$.
For every $\bell\in L$, $F_{\bell}$ lies in the quotient ideal $\fkj:G_d^{\infty}$.
In other words, the ideal $\fki_{\sigma}$ vanishes in $k[\uZ]_{G_d}/\fkj$.
\item [(ii)]
Let $ \bn \in \sigma\cap N$ and $ \fka_{\bn } $ be the ideal
  $\ide{\{Z_{i,s_{i}}\,:\,1\le i\le h, 0\le s_{i}<\acc{\bn}{\bm_i}\}}$ of 
  $k[\uZ_{\infty}] $. Let $ G_{\bn }:=\prod\limits_{i=1}^{d}Z_{i,\ide{\bn ,\bm_i}} $.
Then the ideals $[\fki_{\sigma}]+\fka_{\bn}$ and  $[\fkj]+\fka_{\bn}$ coincide in the localization $k[\uZ_{\infty}]_{G_{\bn}}$.
\end{enumerate}
\end{lemma}
\begin{proof}
Set $G_h:=\prod_{i=1}^hZ_i$. Since $\{\bell_q\,:\,d+1\le q\le h\}$ spans the lattice $L$,
\cite[Lemma 12.2]{Stu96} shows that $\fki_{\sigma}$
vanishes in $k[\uZ]_{G_h}/\fkj$. But \eqref{eq:F_q} shows
that the natural morphism $k[\uZ]_{G_d}\to k[\uZ]_{G_h}$
induces an isomorphism $k[\uZ]_{G_d}/\fkj\cong k[\uZ]_{G_h}/\fkj$.
This shows (i).

By (i) and lemma \ref{lemm:loc:fra}, in the localization
$k[\uZ_{\infty}]_{G_{\bn}}$, the ideal $[\fki_{\sigma}]$ is contained in $[\fkj]+\fka_{\bn}$.
Since the inclusion $[\fkj]\subset [\fki_{\sigma}]$ holds by definition, one deduces that (ii) also holds.
\end{proof}

\section{Technical machinery for computing the formal neighborhood at the generic point of the Nash set}
\label{sec:technical:formal:neighbourhood}
In this section we develop the technical results which we will use in
section \ref{sec:toric:varieties} to obtain a convenient presentation of the
formal neighborhood of the generic point of the Nash set
associated with a divisorial toric valuation. 
The main result of this section is theorem
\ref{theo:main:technical}, whose hypotheses are formulated in a
somewhat abstract form. In section \ref{sec:toric:varieties}
we will verify that these hypotheses hold in the toric setting.

\subsection{} 
We first state a version of the Hensel's lemma for an arbitrary
set of variables. The proof is basically the same as in the case of a finite
set of variables. Since we have not been able to find a convenient
reference, we include it.
\begin{proposition}\label{prop:hensel}
Let $(\cA,\fkM_{\cA})$ be a complete local ring with residue field $\kappa$. Let $I$ be a set
and $\uY=\{Y_{i}\}_{i\in I}$  be a collection of indeterminates. 
Let $J$ be a set and $\{F_j;\,j\in J\}$ be a collection of elements in $\cA[\uY]$.
For $\uy\in \cA^{I}$, we denote by $\uJ_{\uy}$ the $\cA$-linear application
$\cA^I\to \cA^J$ induced by the Jacobian matrix $[\partial_{Y_i}F_j]|_{\uY=\uy}$,
and by $\uF|_{\uY=\uy}\in \cA^J$ the $J$-tuple $(F_j|_{\uY=\uy};\,j\in J)$.

We assume that there exists $\uy^{(0)}\in \cA^I$ such that:
\begin{enumerate}
\item
One has $\uF|_{\uY=\uy^{(0)}}=0 \pmod{\fkM_{\cA}}$.
\item
The $\kappa$-linear application $\kappa^I\to \kappa^J$
deduced from $\uJ_{\uy^{(0)}}$ by reduction modulo $\fkM_{\cA}$ is invertible.
\end{enumerate}
Then there exists a unique element $\cuY=(\cY_i)\in \cA^I$ such that:
\begin{enumerate}
\item
One has $\uF|_{\uY=\cuY}=0$.
\item
For every $i\in I$, one has $\cY_i=y_i^{(0)}\pmod{\fkM_{\cA}}$.
\end{enumerate}
\end{proposition}
\begin{proof}
We begin with two remarks.

First, note that though in this context the Jacobian matrix  may have
an infinite number of rows and columns, each row has only a finite
number of nonzero entries, thus $\uJ_{\uy}$ is well defined for any
$\uy$ in $\cA^I$.
Also, by assumption,
there exists an $\cA$-linear application
\[\uK_{\uy^{(0)}}\colon \cA^J\to \cA^I\]
such that $\uK_{\uy^{(0)}}\uJ_{\uy^{(0)}}=\Id_{\cA^I}\pmod{\fkM_{\cA}}$ and 
$\uJ_{\uy^{(0)}}\uK_{\uy^{(0)}}=\Id_{\cA^J}\pmod{\fkM_{\cA}}$.

Second, note that by the Taylor formula, for $\uy\in \cA^I$, there exists a family $\{\uH_{i_1,i_2}\colon
i_1,i_2\in I\}$ of elements of $\cA[\uY]^J$, depending on $\uy$ and
the $F_j$'s, such that for every $j\in J$, $H_{i_1,i_2,j}=0$ for all
but finitely many $(i_1,i_2)$ and for every $\uz\in \cA^I$  one has
\begin{equation}\label{equa:taylor}
\uF|_{\uY=\uy+\uz}=\uF|_{\uY=\uy}+\uJ_{\uy}(\uz)+\sum_{i_1,i_2\in I}z_{i_1}z_{i_2}\uH_{i_1,i_2}|_{\uY=\uy+\uz}.
\end{equation}
Note that here and elsewhere the notation we use is a condensed form
for writing a possibly infinite number of relations, each of them
being easily verified.

We show by induction that for every $e\geq 0$, there exists
a family $\uy^{(e)}=(y_i^{(e)};\,i\in I)$ of elements of $\cA$, unique modulo $\fkM_\cA^{e+1}$, such 
that $\uy^{(e)}=\uy^{(0)}\pmod{\fkM_{\cA}}$ and $\uF|_{\uY=\uy^{(e)}}=0 \pmod{\fkM_\cA^{e+1}}$.
The case $e=0$ is given by our assumptions.

Now take $e\in \N$ and assume that our induction statement holds for $e$.
Consider the equation
\begin{equation}\label{equa:hensel}
\uF|_{\uY=\uy^{(e)}+\uz}=0 \pmod{\fkM_\cA^{e+2}}
\end{equation}
with unknown $\uz=(z_i)\in \cA^I$ such that $\uz=0\pmod{\fkM_\cA^{e+1}}$.
Since $\uy^{(e)}=\uy^{(0)}\pmod{\fkM_{\cA}}$, the Jacobian matrices
$[\partial_{Y_i}F_j]|_{\uY=\uy^{(0)}}$ and $[\partial_{Y_i}F_j]|_{\uY=\uy^{(e)}}$
are equal modulo $\fkM_{\cA}$. Since $\uz=0\pmod{\fkM_\cA^{e+1}}$, one thus has
\[
\uJ_{\uy^{(e)}}(\uz)=\uJ_{\uy^{(0)}}(\uz)\pmod{\fkM_\cA^{e+2}}.
\]
Thus by \eqref{equa:taylor} and using again $\uz=0\pmod{\fkM_\cA^{e+1}}$,
equation \eqref{equa:hensel} is equivalent to 
\begin{equation}\label{equa:hensel:new:form}
\uJ_{\uy^{(0)}}(\uz)=-\uF|_{\uY=\uy^{(e)}}\pmod{\fkM_\cA^{e+2}}.
\end{equation}
By assumption, $\uF|_{\uY=\uy^{(e)}}=0\pmod{\fkM_\cA^{e+1}}$.
Thus by the first remark above, and using $\uz=0\pmod{\fkM_\cA^{e+1}}$
one more time, \eqref{equa:hensel:new:form} is equivalent to 
\[
\uz=-\uK_{\uy^{(0)}}(\uF|_{\uY=\uy^{(e)}})\pmod{\fkM_\cA^{e+2}}.
\]
Since $\uF|_{\uY=\uy^{(e)}}=0\pmod{\fkM_\cA^{e+1}}$, the latter
expression gives indeed a solution $\uz$ such that $\uz=0\pmod{\fkM_\cA^{e+1}}$.

In order to show the uniqueness of the solution modulo $\fkM_\cA^{e+2}$,
note that if $\uw\in \cA^I$ is such that $\uw=0\pmod{\fkM_\cA^{e+1}}$,
one has by \eqref{equa:taylor}
\[
\uF|_{\uY=\uy^{(e)}+\uw}=\uF|_{\uY=\uy^{(e)}}+\uJ_{\uy^{(e)}}(\uw)\pmod{\fkM_\cA^{e+2}}
\]
thus
\[
\uK_{\uy^{(0)}}(\uF|_{\uY=\uy^{(e)}+\uw})=\uK_{\uy^{(0)}}(\uF|_{\uY=\uy^{(e)}})+\uK_{\uy^{(0)}}(\uJ_{\uy^{(e)}}(\uw))\pmod{\fkM_\cA^{e+2}}
\]
and finally
\[
\uK_{\uy^{(0)}}(\uF|_{\uY=\uy^{(e)}+\uw})=\uK_{\uy^{(0)}}(\uF|_{\uY=\uy^{(e)}})+\uw\pmod{\fkM_\cA^{e+2}}.
\]
\end{proof}

\subsection{} \label{ssec:hypotheses}
We consider the following general setting and notation for the rest of
this section.
Let $ A $ be a $ k $-algebra which is a domain.
Let $ \Omega $ be a finite set, $I$ be a set,
$\uX=\{X_{\omega}\}_{\omega\in\Omega} $ and $\uY=\{Y_{i};\,i\in I\}$ be collections
of indeterminates. Set 
\[\AX:=A[\{X_{\omega}\}_{\omega\in\Omega}]
\quad\text{and}
\quad
\AXY:=A[\{X_{\omega}\}_{\omega\in\Omega},\{Y_{i};\,i\in I\}].
\]
We denote by $\oX{}$ the prime ideal $\ide{X_{\omega};\,\omega\in\Omega}$ of $\AX$.
In accordance with \S \ref{ssec:extension ideals}, for any $\AX$-algebra $B$, we often still denote by
$\oX{}$ the extension of the ideal $\oX{}$ to $B$.

\subsection{}
The following lemma will be useful in the proof of theorem \ref{theo:main:technical}.
\begin{lemma}\label{lemm:pres}
Assume that we are in the setting described in \S\ref{ssec:hypotheses}.
Let $\fkh$ be an ideal of $ \AXY $ such that:
\begin{itemize}
\item[(i)] One has $\ide{\uX}+\fkh=\ide{\uX,\uY}$.
\end{itemize}
Assume moreover that there exists an $\AX$-algebra morphism
$\compl{\eps}\colon \AXY\to \Frac(A)\dblbr{\uX}$ such that:
\begin{itemize}
\item[(ii)] For every $i\in I$ one has $\compl{\eps}(Y_{i})=Y_{i}\pmod{\fkh}$
in the ring $\Frac(A)\dblbr{\uX}[\uY]$.
\item[(iii)]
For every $i\in I$, one has $\compl{\eps}(Y_i)\in \ide{\uX}$.
\end{itemize}
Then the $\oX{}$-adic completion of the localization $(\AXY/\fkh)_{\oX{}}$
is isomorphic to $\Frac(A)\dblbr{\uX}/\compl{\eps}(\fkh)$. 
\end{lemma}
\begin{remark}\label{rema:comp:abs}
Assume that the hypotheses of the lemma hold. Let $\fkg$ be any ideal
    of $ \AXY $ containing $\fkh$ such that $\ide{\uX}+\fkh=\ide{\uX}+\fkg$
  and $\compl{\eps}(\fkh)=\compl{\eps}(\fkg)$.
  Then $\fkg$ also satisfies the hypotheses of the lemma, with the same
    morphism $\compl{\eps}$. In
    particular the lemma shows that the $\oX{}$-adic completions of $(\AXY/\fkh)_{\oX{}}$ and $(\AXY/\fkg)_{\oX{}}$
    are isomorphic.
\end{remark}
\begin{proof}
Note that (iii) shows that $\compl{\eps}(\fkh)$ is contained
in $\ide{\uX}$, thus $\Frac(A)\dblbr{\uX}/\compl{\eps}(\fkh)$ is
a complete noetherian local ring with maximal ideal $\ide{\uX}$.
Moreover (i) and the fact that $A$ is a domain show that $\oX{}$ is indeed a prime ideal of $\AXY/\fkh$.

Let $e\geq 1$. Let $\pi_e$ be the composition of $\compl{\eps}$ with
the quotient morphism
\[
\Frac(A)\dblbr{\uX}/\compl{\eps}(\fkh)\to \Frac(A)[\uX]/(\compl{\eps}(\fkh)+\ide{\uX}^e).
\]
Thanks to (iii), any element of $A[\uX,\uY]$ whose constant term is not zero is sent by
$\compl{\eps}$ to an invertible element of $\Frac(A)\dblbr{\uX}$.
Thus $\pi_e$ induces a morphism
\[
A[\uX,\uY]_{\ide{\uX,\uY}}\to \Frac(A)[\uX]/(\compl{\eps}(\fkh)+\ide{\uX}^e)
\]
which in turn induces a morphism
\[
\wt{\pi}_e\colon A[\uX,\uY]_{\ide{\uX,\uY}}/(\fkh+\ide{\uX}^e)\to \Frac(A)[\uX]/(\compl{\eps}(\fkh)+\ide{\uX}^e).
\]
Note that since $\fkh+\ide{\uX}=\ide{\uX,\uY}$, one has $\fkh+\ide{\uX}^e=\fkh+\ide{\uX,\uY}^e$.
Thus in order to obtain the claimed isomorphism, it suffices to show that $\wt{\pi}_e$ is an isomorphism
for any $e \geq 1$. Since the natural inclusion $A[\uX,\uY]_{\ide{\uX,\uY}}\subset \Frac(A)[\uX,\uY]_{\ide{\uX,\uY}}$ is an isomorphism,
surjectivity is clear.

Let us show injectivity. This amounts to show that if $P\in \Frac(A)[\uX,\uY]$ lies in $\Ker(\pi_e)$,
then $P\in \fkh+\ide{\uX}^e$.
By assumption (ii), for any $P\in A[\uX,\uY]$, one has
$\compl{\eps}(P)=P\pmod{\fkh}$ in the ring $\Frac(A)\dblbr{\uX}[\uY]$.
In particular in the ring $\Frac(A)[\uX,\uY]$ one has 
\[
\compl{\eps}(P)+\ide{\uX}^e=P+\fkh+\ide{\uX}^e
\quad \text{and}\quad \compl{\eps}(\fkh)+\ide{\uX}^e\subset \fkh+\ide{\uX}^e.
\]
Now if $P\in \Frac(A)[\uX,\uY]$ lies in $\Ker(\pi_e)$, then one has
$\compl{\eps}(P)+\ide{\uX}^e \subset \compl{\eps}(\fkh)+\ide{\uX}^e$.
Therefore, by the above properties, one has $P+\fkh+\ide{\uX}^e\subset \fkh+\ide{\uX}^e$.
Thus $P\in \fkh+\ide{\uX}^e$. That concludes the proof.
\end{proof}

Now we can state and prove the main result of the section.
\begin{thm}\label{theo:main:technical}
Assume that we are in the setting described in \S\ref{ssec:hypotheses};
we assume moreover that the set $I$ is of the shape $\Gamma\times \N$ where $\Gamma$ is a
finite set.

Let $ \fkh $ be an ideal of $ \AXY $ such that:
\begin{itemize}\renewcommand{\labelitemi}{$\bullet$}
\item [(A)] The ideal $\fkh $ contains a collection of elements $\{H_{\gamma,s}\;,\gamma\in\Gamma,s\in\N\}$ of the form
  $H_{\gamma,s}= Y_{\gamma,s}U_{\gamma,s}+E_{\gamma,s}$ such that
for every $\gamma\in\Gamma$ and every $ s\in\N $:
\begin{itemize}
\item [(A1)] $ U_{\gamma,s}$ is a unit in $A$.
\item [(A2)] There exists a family $(E_{\gamma,s,r})\in \AX^{\N\cup\{-1\}}$ such that
$E_{\gamma,s,r}\in \oX{}$ for $r\ge s$, $E_{\gamma,s,r}=0$ for all but a finite
number of $r$, and one has
\[
E_{\gamma,s}=E_{\gamma,s,-1}+\sum_{r\in \N}E_{\gamma,s,r}\cdot Y_{\gamma,r}.
\]
\end{itemize}
\item [(B)] Let $(y_{\gamma,s})\in A^{\Gamma\times \N}$ be the unique family
  of elements of $A$ such that for every $\gamma\in \Gamma$
  and $s\in \N$, one has $H_{\gamma,s}|_{Y_{\gamma,r}=y_{\gamma,r}}=0\pmod{\oX{}}$;
  then the ideal $\oXY{}+\fkh$ is contained
  in the ideal $\ide{\uX}+\ide{Y_{\gamma,s}-y_{\gamma,s};\,(\gamma,s)\in \Gamma\times \N}$.
\end{itemize}
Then there exists an $\AX$-algebra morphism $\compl{\eps}\colon \AXY\to \Frac(A)\dblbr{\uX}$
such that:
\begin{itemize}\renewcommand{\labelitemi}{$\bullet$}
\item[(i)] For every $(\gamma,s)\in \Gamma\times \N$ one has $\compl{\eps}(H_{\gamma,s})=0$.
\item[(ii)] For every $(\gamma,s)\in \Gamma\times \N$, one has $\compl{\eps}(Y_{\gamma,s})=Y_{\gamma,s}\pmod{\fkh}$
in the ring $\Frac(A)\dblbr{\uX}[\uY]$.
\item[(iii)] For every ideal $\fkg$ containing $\fkh$ such that
    $\ide{\uX}+\fkh=\ide{\uX}+\fkg$ and $\compl{\eps}(\fkh)=\compl{\eps}(\fkg)$, the $\oX{}$-adic completion of the localization
      $(\AXY/\fkg)_{\oX{}}$ is $\oX{}$-adically isomorphic to $\Frac(A)\dblbr{\uX}/\compl{\eps}(\fkh)$.

\end{itemize}
Assume moreover that:
\begin{itemize}
\item [(C)] For every $ \gamma\in\Gamma $, one has $E_{\gamma,0,-1}\in \AX\setminus \oX{}$.
\end{itemize}
Then one has in addition:
\begin{itemize}
\item[(iv)] For every $\gamma\in \Gamma$, $\compl{\eps}(Y_{\gamma,0})$ is a unit in $\Frac(A)\dblbr{\uX}$.
\end{itemize}
\end{thm}

\begin{proof}
First note that for each $\gamma\in \Gamma$, the reduction of the $H_{\gamma,s}$'s
modulo $\oX{}$ gives a triangular and invertible $A$-linear system in
the $Y_{\gamma,s}$'s.
Thus  the existence and uniqueness of
$(y_{\gamma,s})$ in assumption (B) is a straightforward consequence of assumption (A).
In fact, up to dividing $H_{\gamma,s}$ by $U_{\gamma,s}$
and modifying the  $E_{\gamma,s,r}$'s, one may assume that for every $\gamma,s,r$ one
  has $E_{\gamma,s,r}\in \ide{\uX}$ and that for every $\gamma,s$ one has
\begin{equation}\label{equa:Hgammas:red}
H_{\gamma,s}=Y_{\gamma,s}-y_{\gamma,s}+\sum_{r=0}^{s-1}\alpha_{\gamma,r}(Y_{\gamma,r}-y_{\gamma,r})+E_{\gamma,s,-1}+\sum_{r\in \N}E_{\gamma,s,r}(Y_{\gamma,r}-y_{\gamma,r}),
\end{equation}
where the $\alpha_{\gamma,r}$'s are elements of $A$.

Applying proposition \ref{prop:hensel} with
$\cA=\Frac(A)\dblbr{\uX}$ and $\{F_j;\,j\in J\}=\{H_{\gamma,s};\,(\gamma,s)\in \Gamma\times \N\}$, this shows the existence of a family
$\{\cY_{\gamma,s};\,\gamma\in \Gamma,\,s\in \N\}$ of elements of $\Frac(A)\dblbr{\uX}$
such that for every $(\gamma,s)\in \Gamma\times \N$
one has $\cY_{\gamma,s}=y_{\gamma,s}\pmod{\ide{\uX}}$ and $H_{\gamma,s}|_{Y_{\gamma,r}=\cY_{\gamma,r}}=0$. Thus mapping
$Y_{\gamma,s}$ to $\cY_{\gamma,s}$ defines an $\AX$-algebra morphism $\compl{\eps}\colon \AXY\to \Frac(A)\dblbr{\uX}$
such that (i) holds.

For every $\gamma\in \Gamma$, \eqref{equa:Hgammas:red} and an induction on $s$ shows
that for every $s$ one has $Y_{\gamma,s}-y_{\gamma,s}\in \ide{\uX}+\fkh$. By assumption (B), one then has
$\oXY{}+\fkh=\ide{\uX}+\ide{Y_{\gamma,s}-y_{\gamma,s}\,;\,(\gamma,s)\in \Gamma\times \N}$.
Thus $\Frac(A)\dblbr{\uX}[\uY]/\fkh$
is a noetherian local ring with maximal ideal $\ide{\uX}$.

On the other hand, \eqref{equa:Hgammas:red} shows that
for every $(\gamma,s)\in \Gamma\times \N$, since $H_{\gamma,s}\in \fkh$ and
$\compl{\eps}(H_{\gamma,s})=0$, one has in the ring $\Frac(A)\dblbr{\uX}[\uY]$
the relation
\[
\cY_{\gamma,s}-Y_{\gamma,s}=-\sum_{r=0}^{s-1}\alpha_{\gamma,r}(\cY_{\gamma,r}-Y_{\gamma,r})-\sum_{r\geq 0}E_{\gamma,s,r}(\cY_{\gamma,r}-Y_{\gamma,r})\pmod{\fkh}.
\]
Thus by a straightforward induction one gets that
$\cY_{\gamma,s}-Y_{\gamma,s}\in \ide{\uX}^e+\fkh$ for every $\gamma,s$ and $e\geq 1$, and
finally by the Krull's intersection theorem
$\cY_{\gamma,s}-Y_{\gamma,s}\in \fkh$ for every $\gamma,s$.
Thus (ii) holds.
Recalling that $\cY_{\gamma,s}-y_{\gamma,s}\in \ide{\uX}$,
(iii) then follows from an application of lemma \ref{lemm:pres} (replacing $Y_{\gamma,s}$ with $Y_{\gamma,s}-y_{\gamma,s}$)
and remark \ref{rema:comp:abs}.

Assumption (C) is equivalent to the property $y_{\gamma,0}\in A\setminus \{0\}$.
Then $y_{\gamma,0}$ is a unit in $\Frac(A)\dblbr{\uX}$,
and since $\cY_{\gamma,0}=y_{\gamma,0}\pmod{\ide{\uX}}$,
$\cY_{\gamma,0}=\compl{\eps}(Y_{\gamma,0})$ also is a unit, and (iv) holds.
\end{proof}
\begin{remark}\label{rema:fkh:generated}
In the statement of the theorem, if one assumes that (A) holds and
that moreover $\fkh$ is generated by the $H_{\gamma,s}$'s and some
elements of $\ide{\uX}$, then (B) automatically
holds. Indeed, the above proof shows that without changing the ideal
generated by the $H_{\gamma,s}$'s, one may assume  that
for every $\gamma,s$ one has $H_{\gamma,s}=Y_{\gamma,s}-y_{\gamma,s}\pmod{\ide{\uX}}$.
\end{remark}

\section{Technical machinery for the comparison theorem}\label{sec:technical:comparison}
In this section we will obtain the crucial technical result 
(theorem \ref{thm:comparison:technical}) allowing to establish our comparison theorem 
in section \ref{sec:toric:varieties}. 
As for theorem \ref{theo:main:technical}
the hypotheses are formulated in a
  somewhat abstract form, and in section \ref{sec:toric:varieties}
  we will verify that these hypotheses hold in the toric setting.

\subsection{} \label{hyp:reduction:and:comparison}
We begin with an elementary yet useful lemma.
\begin{lemma}\label{lemm:inclusion:ideals}
Let $K$ be a field, $\cA$ be an object of $\Lacp_K$, $\fka$ and $\fkb$ two ideals of $\cA$
such that for every object $\cB$ of $\Lacp_K$ one has the inclusion
\[
\{\phi\in \Hom_{\Lacp_K}(\cA,\cB),\quad \fkb\subset\Ker(\phi)\}\subset\{\phi\in \Hom_{\Lacp_K}(\cA,\cB),\quad \fka\subset\Ker(\phi)\}
\]
Then one has the inclusion $\fka\subset \fkb$.
\end{lemma}
\begin{proof}
We apply the assumption with $\phi$ the quotient morphism $\cA\to \cA/\fkb$.
\end{proof}

\begin{notation}\label{not:homogeneous:coefficients}
Let $ \Delta $ be a finite set and $\uY$ be the set of indeterminates $\{Y_{\delta};\delta\in\Delta\}$.
Let $R$ be a ring.
Let $\cuY(\vart):=\{\cY_{\delta}(\vart):\delta\in\Delta\}$ be a family of elements in
the power series ring $R\dbT$.
Let $P\in R[\uY]$.
Then we define the family 
$\left\{P_{s,\cuY(\vart)}:s\in\N\right\}$
of elements of $R$ by the following equality in $R\dbT$:
\begin{align}\label{eqn:definition homogeneous coefficients}
P|_{Y_{\delta}=\cY_{\delta}(\vart)}=\sum\limits_{s\in\N}P_{s,\cuY(\vart)}\vart^{s}.
\end{align}
\end{notation}
\begin{remark}\label{rema:definition:homogeneous:coefficients}
Keep the same notation as before. Let $S$ be another ring,
$\phi\colon R\to S$ is a ring morphism. We also denote by $\phi$ the
induced morphisms $R[\uY]\to S[\uY]$ and $R\dbT\to S\dbT$ obtained by
applying $\phi$ coefficientwise. Then for every $s\in \N$ one has $\phi(P_{s,\cuY(\vart)})=\phi(P)_{s,\phi(\cuY(\vart))}$.
\end{remark}

\subsection{}\label{hyp:reduction:and:comparison:2}
Now we can state and prove the main result of the section.
\begin{thm}\label{thm:comparison:technical}
Let $ K $ be a field extension of $ k $, $ \Delta$ be a
finite set and $\uY$ be the set of indeterminates $\{Y_{\delta};\delta\in\Delta\}$.
Let $ (d_\delta)\in\N^{\Delta} $ be a family of nonnegative integers.
Let $\uX$ be the set of variables $\{X_{\delta,j};\delta\in\Delta,\, 0\le j< d_{\delta}\}$.
We denote by $ \oX{} $  the maximal ideal of the power series ring $\KdX$.

Let $ \Omega $ be a (possibly infinite) set, and let 
$\{P_{\omega}\}_{\omega\in\Omega} $ be a family of elements in the
polynomial ring $ K[\uY] $ such that for every $ \omega\in\Omega $,
one has:
\begin{itemize}
\item [(I)] One may write
  $P_{\omega}=\prod\limits_{\delta\in\Delta}Y_{\delta}^{u^+_{\omega,\delta}}-\prod\limits_{\delta\in\Delta}Y_{\delta}^{u^{-}_{\omega,\delta}}$,
  where $ u^+_{\omega,\delta},u^{-}_{\omega,\delta}\in\N$.
\item [(II)] One has
$P_{\omega}|_{Y_{\delta}=\vart^{d_{\delta}}}=0 $ in $ K\brT $, in other words
\[\sum\limits_{\delta\in\Delta}d_{\delta}u^+_{\omega,\delta}=  \sum\limits_{\delta\in\Delta}d_{\delta}u^{-}_{\omega,\delta}=:c_{\omega}.
\]
\end{itemize} 
Let $ \{x_{\delta,j}:\delta\in\Delta,\;j\ge d_{\delta}\} $ be a family of elements in $ \KdX$.
For $\delta\in \Delta$, set
\[
\cY_{\delta}(\vart):=\sum_{j=0}^{d_{\delta}-1}X_{\delta,j}\vart^j+\sum_{j\ge  d_{\delta}}x_{\delta,j}\vart^j\in \KdX\dbT
\]
and
\[
\wt{\cY_{\delta}}(\vart):=\sum_{j=0}^{d_{\delta}-1}X_{\delta,j}\vart^j+\vart^{d_{\delta}}\in \KdX\brT
\]
We assume:
\begin{itemize}
\item[(a)] for every $\delta\in\Delta$, $ x_{\delta,d_{\delta}}$ is a unit;
\item[(b)] for every $ \omega\in\Omega $ and every $ s\ge c_{\omega}$, one has $P_{\omega,s,\cuY(\vart)}=0$.
\end{itemize}

\noindent
We consider the following ideals of $\KdX$ :
$\fka:=\ide{\left\{P_{\omega,s,\cuY(\vart)}\,:\,\omega\in\Omega,\,s\in\N\right\}}$
 and $\fkb:=\left\ide{\left\{P_{\omega,s,\wt{\cuY}(\vart)}\,:\,\omega\in\Omega,\,s\in\N\right\}\right}$.

Then $\KdX/\fka$ and $\KdX/\fkb$ are isomorphic objects of $\Lacp_K$.
\end{thm}

\begin{proof}
By assumption (a), for every $\delta\in \Delta$,  the series
$\cY_{\delta}(\vart)$ is a $d_{\delta}$-regular element of $\KdX\dbT$.
Thus by the Weierstrass preparation theorem (see \S \ref{subsec:WPT}),
there exists a family
$ \{\fkX_{\delta,j}:\delta\in\Delta,\, 0\le j<d_{\delta}\} $
of elements of the maximal ideal $ \oX{} $ of $ \KdX $
and a family $\{U_{\delta,r}:\delta\in\Delta,\,r\in \N\}$ of elements
of $\KdX$ with $U_{\delta,0}$ an unit, such that, setting
\[
W_{\delta}(\vart):=\vart^{d_{\delta}}+\sum\limits_{j=0}^{d_{\delta}-1}\fkX_{\delta,j}\vart^{j}
\quad \text{and} \quad U_{\delta}(\vart):=\sum\limits_{r\in\N}U_{\delta,r}\vart^{r}
\]
one has
\begin{equation}\label{eqn:Weierstrass decomposition}
\cY_{\delta}(\vart)=W_{\delta}(\vart)U_{\delta}(\vart).
\end{equation}
Identifying the $\vart$-coefficients in the latter equation
yields the following relations in $\KdX$:
\[X_{\delta,j}=\fkX_{\delta,j}U_{\delta,0}+\sum\limits_{r=0}^{j-1}\fkX_{\delta,r}U_{\delta,j-r},\quad
0\le j < d_{\delta}.
\]
Since $U_{\delta,0}$ is a unit, we deduce that the
element of $\Hom_{\Lacp_K}(\KdX,\KdX)$
sending $ X_{\delta,j} $ to $ \fkX_{\delta,j} $ for $ \delta\in\Delta$
and $ 0\le j< d_{\delta} $ is an isomorphism.

Setting
\[
\fkc:=\left\ide{\left\{P_{\omega,s,\{W_{\delta}(\vart)\}}:\omega\in\Omega,\,s\in\N\right\}\right},
\]
the above isomorphism shows that $\KdX\dbT/\fkb$ and $\KdX\dbT/\fkc$ are isomorphic
objects in $\Lacp_K$. To conclude the proof, we show that $\fka=\fkc$,
using lemma \ref{lemm:inclusion:ideals}.

Let $(\cB,\fkm_{\cB})$ be an object in $\Lacp_K$, and let $\phi$ be an
element of $\Hom_{\Lacp_K}(\KdX,\cB)$. We still denote by $\phi$ the
induced morphism $\KdX\dbT\to\cB\dbT$ obtained by applying $\phi$ coefficientwise.

Let us assume that for every $\omega\in\Omega $
one has $P_{\omega}|_{Y_{\delta}=\phi(\cY_{\delta}(\vart))}=0$. One has to show that 
for every $ \omega\in\Omega $ one has $P_{\omega}|_{Y_{\delta}=\phi(W_{\delta}(\vart))}=0$. 

From our assumption and hypothesis (I) we deduce the following equality in $ \cB\dbT $:
\[
\prod\limits_{\delta\in\Delta}
\phi(\cY_{\delta}(\vart))^{u^+_{\omega,\delta}}
=\prod\limits_{\delta\in\Delta}  \phi(\cY_{\delta}(\vart))^{u^{-}_{\omega,\delta}}
\]
which can be rewritten, using equation \ref{eqn:Weierstrass decomposition}, as 
\begin{equation}\label{equa:Wdelta:Udelta}
\prod\limits_{\delta\in\Delta}
\phi(W_{\delta}(\vart))^{u^+_{\omega,\delta}}
\prod\limits_{\delta\in\Delta}
\phi(U_{\delta}(\vart))^{u^+_{\omega,\delta}}
=\prod\limits_{\delta\in\Delta}\phi(W_{\delta}(\vart))^{u^{-}_{\omega,\delta}}
\prod\limits_{\delta\in\Delta}
\phi(U_{\delta}(\vart))^{u^{-}_{\omega,\delta}}.
\end{equation}
Note that for every $\delta\in \Delta$, $\phi(W_{\delta}(\vart))$ is
a Weierstrass polynomial of degree $d_{\delta}$ in $\cB\dbT$
and $\phi(U_{\delta}(\vart))$ is a unit in $\cB\dbT$, since $\phi(U_{\delta,0})$ is.

By uniqueness of the Weierstrass factorization in $\cB\dbT$, one gets the equality
\begin{equation}\label{equa:Wdelta}
\prod\limits_{\delta\in\Delta}\phi(W_{\delta}(\vart))^{u^+_{\omega,\delta}}
=\prod\limits_{\delta\in\Delta}\phi(W_{\delta}(\vart))^{u^{-}_{\omega,\delta}}
\end{equation}
which means exactly that $P_{\omega}|_{Y_{\delta}=\phi(W_{\delta}(\vart))}=0$.

Conversely, assume that for every $\omega\in\Omega$ one has
$P_{\omega}|_{Y_{\delta}=\phi(W_{\delta}(\vart))}=0$, in other words, 
that \eqref{equa:Wdelta} holds, and let us show that 
for every $\omega\in\Omega$ one has $P_{\omega}|_{Y_{\delta}=\phi(\cY_{\delta}(\vart))}=0$. Let
$\wt{W}_{\omega}(\vart)\in \cB\dbT$ be the common value of both members of \eqref{equa:Wdelta}.
Note that $\wt{W}_{\omega}(\vart)$ is a Weierstrass polynomial of degree $c_{\omega}$.
On the other hand, one has
\[
P_{\omega}|_{Y_{\delta}=\phi(\cY_{\delta}(\vart))}
=
\wt{W}_{\omega}(\vart)
\left(\prod\limits_{\delta\in\Delta}
\phi(U_{\delta}(\vart))^{u^+_{\omega,\delta}}
-
\prod\limits_{\delta\in\Delta}
\phi(U_{\delta}(\vart))^{u^{-}_{\omega,\delta}}\right).
\]
By assumption (b), $P_{\omega}|_{Y_{\delta}=\phi(\cY_{\delta}(\vart))}$ is an
element of the polynomial ring $\cB\brT$ with degree less than $c_{\omega}$.
By the uniqueness of the Weierstrass division by $\wt{W}_{\omega}(\vart)$ in $\cB\dbT$
one concludes that $P_{\omega}|_{Y_{\delta}=\phi(\cY_{\delta}(\vart))}=0$.
\end{proof}

\section{A comparison theorem between formal neighborhoods}
\label{sec:toric:varieties}
In this section we will make use of the results in sections
\ref{sec:technical:formal:neighbourhood} and
\ref{sec:technical:comparison} in order to obtain the main comparison
theorem as an application of the results in those sections to the
toric setting. It should be noted that our results provide basically
two approaches for computing effectively the formal neighborhood of the
generic point of the Nash set associated with a toric valuation. The
first one is based on an effective implementation of the Hensel's lemma
crucially used in section \ref{sec:technical:formal:neighbourhood}. The second one takes
advantage of the comparison theorem in order to use exactly the same techniques
as in the case of rational arcs described in \cite{Bou-Seb:toric}. The
latter seems to be much more efficient in practice. See \S \ref{subsec:exam}
below for an explicit example of computation, as well as
\cite{MMC:PhD} for more details and explicit examples.

\subsection{}\label{ssec:differential structure}
We retain the notation introduced in section \ref{sec:recoll:arc:toric}.
In particular, $V_{\sigma}$ is the affine toric $k$-variety of
dimension $d$ associated with a cone $\sigma$ and presented as $k[\uZ]/\fki_{\sigma}$,
where $\uZ=\{Z_i\,:\,i\in\{1,\dots,h\}\}$ and $\fki_{\sigma}$
is generated by the binomial elements $\{F_{\bell}=\uZ^{\ell^+}-\uZ^{\ell^-}\,;\,\bell\in L\}$, 
$L$ being a subgroup of $\Z^h$.
Moreover, denoting by $\uZ_{\infty}$ the set of variables $\{Z_{i,s}\,:\,i\in\{1,\dots,h\}, s\in\N\}$,
the arc scheme $\arc(V_{\sigma})$ associated with the affine toric variety
$V_{\sigma} $ may be identified with the affine scheme
$\Spec\left(k[\uZ_{\infty}]/[\fki_{\sigma}]\right)$; the ideal
$ [\fki_{\sigma}] $ is generated by the elements $\{F_{\bell,s}\,;\,\bell\in L, s\in\N \} $ 
where for every $\bell\in L$ the elements $F_{\bell,s}\in k[\uZ_{\infty}]$
may be characterized by the following equality in $k[\uZ_{\infty}]\dbT$:
\[
F_{\bell}|_{Z_i=\sum\limits_{s\in\N}Z_{i,s}\vart^{s}}=\sum\limits_{s\in\N}F_{\bell,s}\vart^s.
\]

\subsection{}\label{ssec:arcs et point stable}
The following proposition gives an explicit description of the generic
point of the Nash set associated with a toric valuation.
\begin{proposition}\label{prop:ideal:stable:point}
Let $ \bn \in \sigma\cap N$ be a toric valuation, $\mds_{\bn}$ 
be the associated Nash set 
and $\eta_{\bn }$ be the generic point of $\mds_{\bn}$.
Let $ \fka_{\bn } $ be the ideal
$\ide{\{Z_{i,s_{i}}\,:\,1\le i\le h, 0\le s_{i}<\acc{\bn}{\bm_i}\}}$ of 
$k[\uZ_{\infty}] $. Let $ G_{\bn }:=\prod\limits_{i=1}^{d}Z_{i,\ide{\bn ,\bm_i}} $ and $g_{\bn}$ be the image of $G_{\bn}$
in $\struc{\arc(V_{\sigma})}$.
Then:
\begin{itemize}
\item [(i)]  The prime ideal of $ \struc{\arc(V_{\sigma})} $ corresponding
with $\eta_{\bn } $ is the radical of the image of $\fka_{\bn }$ in $\struc{\arc(V_{\sigma})}$.
\item [(ii)] The point $ \eta_{\bn } $ belongs to the distinguished
  open subset $\diso{g_{\bn }}$ of $ \arc(V_{\sigma})$. The prime ideal of
  $\struc{\arc(V_{\sigma})}_{g_{\bn}}$ corresponding with $\eta_{\bn } $ is the extension
  of $\fka_{\bn }$ to $\struc{\arc(V_{\sigma})}_{g_{\bn}}$.
\end{itemize}
\end{proposition}

\begin{proof}
Assertion (i) follows from lemma \ref{lemm:decr:mds:toric:set}.

Let us now prove assertion (ii). By (i), it is enough to show that
the $k$-algebra $R:=k[\uZ_{\infty}]_{G_{\bn}}/[\fki_{\sigma}]+\fka_{\bn}$
is a domain. Let us show that its functor of points is isomorphic to
the functor of points of the $k$-algebra $\struc{\arc(\torus)}$, the latter
being a domain since $\torus$ is a smooth irreducible variety.

Let $ A  $ be a $ k $-algebra. By the very definition of $R$ and
\S\ref{ssec:jetarc} and \ref{ssec:toric:basic},  the
set $\Hom_{\Alg_k}(R,A )$ is in natural bijection with the set of
semigroup morphisms $\phi\colon S_{\sigma}\to A\dbT$ such that for
$1\le i\le h$ one has $\ord_{t}(\phi(\bm_i))\ge \acc{\bm_i}{\bn}$ and for $1\le i\le d$
one has $\vart^{-\acc{\bm_i}{\bn}}\phi(\bm_i)\in (A\dbT)^{\inv}$. The latter
property also holds for $d+1\le i \le h$ since $\bm_i$ is a $\Z$-basis of $M$. In
particular $\bm\mapsto \vart^{-\acc{\bm}{\bn}}\phi(\bm)$ is a semigroup morphism
$S_{\sigma}\mapsto (A\dbT)^{\inv}$. Thus the
set $\Hom_{\Alg_k}(R,A )$ is in natural bijection with the set of
group morphism $M\to (A\dbT)^{\inv}$, which in turn is in natural
bijection with $\Hom_{\Alg_k}(\struc{\arc(\torus)},A)$.
\end{proof}

\subsection{}\label{ssec:complete:intersection}
Recall that $\fkj$ is the ideal $\ide{ \Fdis{q}\,;\,d+1\le q\le h }$
  of the ring $ k[\uZ] $. Its defines an affine $k$-scheme
 $W:=\Spec\left(k[\uZ]/\fkj \right) $ which contains $ V_{\sigma} $ as a closed subscheme.
Recall also that $\arc(W)$ may be identified with $\Spec\left(k[\uZ_{\infty}]/ [\fkj]\right)$
and that $[\fkj ]=\ide{\Fdisu{q}{s}\,;\,d+1\le q\le h, s\in\N }$.
The closed immersion $V_{\sigma}\rightarrow W $ induces a closed immersion
$\arc(V_{\sigma})\rightarrow \arc(W)$ between the corresponding arc
schemes. For $\bn\in \sigma\cap N$, let $\eta'_{\bn }$ be the image of $ \eta_{\bn } $
by this closed immersion. We shall reduce the computation of the
formal neighborhood of $\arc(V_{\sigma})$ at $\eta_{\bn}$ to that of
the formal neighborhood of $\arc(W)$ at $\eta'_{\bn}$. We will say
that we are in the toric setting in the former situation and (abusing
terminology) in the complete intersection setting in the latter.

The following lemma is a straightforward consequence of the
definition, lemma \ref{lemm:quotient}(ii) and proposition \ref{prop:ideal:stable:point}.
\begin{lemma}\label{lem:complete:intersection}
Retain the notation and hypotheses of proposition \ref{prop:ideal:stable:point}.
Let $g'_{\bn}$ be the image of $G_{\bn}$ in $\struc{\arc(W)}$.

Then the point $ \eta'_{\bn } $ belongs to the distinguished
 open subset $\diso{g'_{\bn }}$ of $ \arc(W)$, and the prime ideal of
 $\struc{\arc(W)}_{g'_{\bn}}$ corresponding with $\eta'_{\bn } $ is the extension
 of $\fka_{\bn }$ to $\struc{\arc(W)}_{g'_{\bn}}$.
 \end{lemma}

\begin{notation}
For $ q\in\{1,\dots,h\} $ and $\br\in\N^q $ we denote
by $\uZ_{\leq q,\leq r_i}$ the set of
variables $\{Z_{i,s_i}\,;\,1\le i\le q,\, 0 \leq s_i\leq  r_i\}$.
If $ q=h $ we write $\uZ_{\bullet, \leq r_i} $ instead of $\uZ_{\leq h, \leq r_i}$.
We define similarly $\uZ_{\leq q,\geq r_i}$, $\uZ_{\geq q,\leq r_i}$ and so on. 
\end{notation}

\subsection{}\label{ssec:toric formal neighbourhood}
The following lemma shows that we can apply theorem \ref{theo:main:technical}
in the complete intersection setting.

\begin{lemma}\label{lem:toric:hyp:formal:neighbourhood}
Let $\bn\in \sigma\cap N $ be a toric valuation.

Let $G_{\bn }:=\prod\limits_{i=1}^{d}Z_{i,\acc{\bn}{\bm_i}}$ and $ A $ be the $k$-algebra $k[\uZ_{\leq d,\ge\acc{\bn}{\bm_i}} ]_{G_{\bn }} $.

Let $\Omega$ be the finite set $\{(i,s_i);i\in\{1,\dots,h\}, 0\le s_i <\ide{\bn ,\bm_i}\}$.
For $\omega\in \Omega$, set $X_{\omega}:=Z_{\omega}$.
Set $\Gamma=\{d+1,\dots,h\}$.
For $q\in \Gamma$ and $s\in \N$, set $Y_{q,s}:=Z_{q,\acc{\bn}{\bm_q}+s}$.

Let $\fkh$ be the extension of the ideal  $[\fkj ]$ in $k[\uZ_{\infty}]_{G_{\bn }}$.
For $s\in \N$ and $q\in \Gamma$, set $H_{q,s}:=\Fdisu{q}{\acc{\bn}{\bell_q}+s}$
(Recall from \S\ref{subsec:toric:ideal} the definition of $\acc{\bn}{\bell_q}$.)
Then with this notation the hypotheses in theorem \ref{theo:main:technical} hold true.
\end{lemma}

\begin{proof}
Note that with the notation of the statement, one has in particular $A[\uX,\uY]=k[\uZ_{\infty}]_{G_{\bn }}$
and the ideal $\oXY{} $ corresponds to $\fka_{\bn }=\ide{\uZ_{\bullet, <\acc{\bn}{\bm_i}}}$.

Let us show that assumption (A) in theorem \ref{theo:main:technical} holds.
Pick up $q\in\{d+1,\dots,h\}$. Set 
\[
\Lambda_{q}^{+}:=\{i\in\{1,\dots,d\};\ell_{q,i}>0\},\quad \Lambda_{q}^{-}:=\{i\in\{1,\dots,d\};\ell_{q,i}<0\}
\]
\[\Theta_{q,s}^{+}:=\{
(r_q,(r_{i,k}\,;\,i\in\Lambda_{q}^{+},1\le k\le\ell_{q,i}))\,;\,r_q,r_{i,k}\in\N,\,
r_q+\sum\limits_{i\in\Lambda_{q}^{+}}\sum\limits_{k=1}^{\ell_{q,i}}r_{i,k}=s\}
\]
\[
\text{and}\quad
\Theta_{q,s}^{-}:=\{
(r_{i,k}\,;\,i\in\Lambda_{q}^{-},1\le k\le-\ell_{q,i})\,;\,r_{i,k}\in\N,
\,\sum\limits_{i\in\Lambda_{q}^{-}}\sum\limits_{k=1}^{-\ell_{q,i}}r_{i,k}=s\}.
\]
Then by \eqref{eq:F_q} and \eqref{equa:HS}, the polynomial $\Fdisu{q}{s}$ has the following form:
\begin{equation}\label{eq:form:equations}
\Fdisu{q}{s}=
\sum\limits_{(r_q,r_{i,k})\in\Theta_{q,s}^{+}}Z_{q,r_q}\prod\limits_{i\in\Lambda_{q}^{+}}\prod\limits_{k=1}^{\ell_{q,i}}Z_{i,r_{i,k}}
-\sum\limits_{(r_{i,k})\in\Theta_{q,s}^{-}}\prod\limits_{i\in\Lambda_{q}^{-}}\prod\limits_{k=1}^{-\ell_{q,i}}Z_{i,r_{i,k}}.
\end{equation}
Note that setting $r_q=\acc{\bn}{\bm_q}+s$ and $r_{i,k}=\acc{\bn}{\bm_i}$
for $1\leq k\leq \ell_{q,i}$  defines an element of $\Theta_{q,\acc{\bn}{\bell_q}+s}^{+}$.
Set
\[
U_q:=\prod\limits_{i\in\Lambda_{q}^{+}}\prod\limits_{k=1}^{\ell_{q,i}}Z_{i,\acc{\bn}{\bm_{i}}}.
\]
By the definition of $G_{\bn }$, 
$U_q$ is an invertible element of $k[\uZ_{\le d,\ge\acc{\bn}{\bm_i}}]_{G_{\bn}}$.

Set
\[
E_{q,s,-1}:=
\sumsubu{(r_q,r_{i,k})\in\Theta_{q,\acc{\bn}{\bell_q}+s}^{+}\\ r_q<\acc{\bn}{\bm_q}}
Z_{q,r_q}\prod\limits_{i\in\Lambda_{q}^{+}}\prod\limits_{k=1}^{\ell_{q,i}}Z_{i,r_{i,k}}
-\sum\limits_{(r_{i,k})\in\Theta_{q,\acc{\bn}{\bell_q}+s}^{-}}\prod\limits_{i\in\Lambda_{q}^{-}}\prod\limits_{k=1}^{-\ell_{q,i}}Z_{i,r_{i,k}}.
\]
For $r\in \N$, set $\delta_{r,s}=1$ if $r=s$ and $0$ otherwise, and
\[
E_{q,s,r}=:
-\delta_{s,r}U_q+
\sumsubu{
(r_{i,k}\,;\,i\in\Lambda_{q}^{+},1\le k\le\ell_{q,i}))\,;\,r_{i,k}\in\N
\\(\acc{\bn}{\bm_q}+r,(r_{i,k}))\in\Theta_{q,\acc{\bn}{\bell_q}+s}^{+}}
\prod\limits_{i\in\Lambda_{q}^{+}}\prod\limits_{k=1}^{\ell_{q,i}}Z_{i,r_{i,k}}.
\]
Thus by \eqref{eq:form:equations},
one has
\[
\Fdisu{q}{\acc{\bn}{\bell_q}+s}=U_qZ_{q,\acc{\bn}{\bm_q}+s}+E_{q,s,-1}+\sum_{r\in \N}E_{q,s,r}Z_{q,\acc{\bn}{\bm_q}+r}.
\]
Since $\Lambda_{q}^{\pm}\subset \{1,\dots,d\}$, it is clear that for
  $r\in \N$ one has $E_{q,s,r}\in k[\uZ_{\leq d,\bullet}]$, and that 
  $E_{q,s,-1}\in k[\uZ_{\leq d,\bullet}\cup \uZ_{\bullet,<\acc{\bn}{\bm_i}}]$.

Thus (A1) is satisfied, and in order to show that (A2) also holds,
it remains to prove that for any $r\ge s$, each monomial of $E_{q,s,r}$
contains a variable $Z_{i,r_i}$ with $i\in \{1,\dots,d\}$ and $r_i<\acc{\bn}{\bm_i}$.
Take $(r_{i,k}\,;\,i\in\Lambda_{q}^{+},1\le k\le\ell_{q,i}))$ a family
of nonnegative integers such that $(\acc{\bn}{\bm_q}+r,(r_{i,k}))\in\Theta_{q,\acc{\bn}{\bell_q}+s}^{+}$, that is
\[
\acc{\bn}{\bm_q}+r+\sum\limits_{i\in\Lambda_{q}^{+}}\sum\limits_{k=1}^{\ell_{q,i}}r_{i,k}=\acc{\bn}{\bell_q}+s.
\]
We have to show that either at least one of the $r_{i,k}$'s is
$<\acc{\bn}{\bm_i}$ or $r=s$ and $r_{i,k}=\acc{\bn}{\bm_i}$ for every
$i,k$. (The latter case corresponds to the monomial
$U_qZ_{q,\acc{\bn}{\bm_q}+s}$.) Assume $r_{i,k}\geq \acc{\bn}{\bm_i}$ for every $i,k$. Then
\[
\acc{\bn}{\bell_q}+s=\acc{\bn}{\bm_q}+r+\sum\limits_{i\in\Lambda_{q}^{+}}\sum\limits_{k=1}^{\ell_{q,i}}r_{i,k}
\ge 
r+\acc{\bn}{\bm_q+\sum\limits_{i\in\Lambda_{q}^{+}}\ell_{k,i}\bm_{i}}=r+\acc{\bn}{\bell_q}.
\]
If $r>s$ this is a contradiction. If $r=s$, the first minoration must
be an equality, which imposes $r_{i,k}=\acc{\bn}{\bm_i}$ for every $i,k$.

Let us prove that (C) holds.
We have to show that $E_{q,0,-1}$ does not belong to the ideal $\ide{\uZ_{\bullet, <\acc{\bn}{\bm_i}}}$.
By the definition of $E_{q,0,-1}$ it is enough to show that
\[
\wt{E_{q,0,-1}}:=
-\sum\limits_{(r_{i,k})\in\Theta_{q,\acc{\bn}{\bell_q}}^{-}}\prod\limits_{i\in\Lambda_{q}^{-}}\prod\limits_{k=1}^{-\ell_{q,i}}Z_{i,r_{i,k}}
\]
does not belong to the ideal $\ide{\uZ_{\bullet, <\acc{\bn}{\bm_i}}}$.
But arguing similarly as above, one sees that the only monomial in
$\wt{E_{q,0,-1}}$ not belonging to the above ideal corresponds to $r_{i,k}=\acc{\bn}{\bm_i}$. Thus 
one has $\wt{E_{q,0,-1}}=\prod\limits_{i\in\Lambda_{q}^{-}}Z_{i,\acc{\bn}{\bm_i}}^{-\ell_{q,i}}\pmod{\ide{\uZ_{\bullet, <\acc{\bn}{\bm_i}}}}$
which allows to conclude.

Let us show that (B) holds. Since $\fkh$ is the extension of the ideal
$[\fkj]$ in $k[\uZ_{\infty}]_{G_{\bn }}$, it is generated by the
union of the families $\{H_{q,s};\,q\in \Gamma,\,s\in \N\}$
and $\{\Fdisu{q}{s};\,q\in \Gamma,\,s\in \N,\,s<\acc{\bn}{\bell_q}\}$.

Arguing similarly as above, one sees using \eqref{eq:form:equations}
that in case $s<\acc{\bn}{\bell_q}$ every monomial of
$\Fdisu{q}{s}$ must contain a variable $Z_{i,r}$ with $r<\acc{\bn}{\bm_i}$.
Thus $\fkh$ is generated by some elements of $\ide{\uX}$ and the $H_{\gamma,s}$'s.
By remark \ref{rema:fkh:generated}, assumption (B) holds in this case.
\end{proof}

\subsection{}\label{ssec:toric series}
Thanks to lemma \ref{lem:toric:hyp:formal:neighbourhood}, we can apply
    theorem \ref{theo:main:technical} in the complete intersection
    setting. In the proof of the following corollary, we shall see that
    this also holds in the toric setting.
\begin{corollary}\label{cor:toric:formal:neighbourhood}\label{cor:toric:reduction:complete:intersection:formal:stable}
Let $ \bn  $ be a toric valuation of $ \sigma\cap N $. 
There exists a $k[\uZ_{\leq d,\ge\acc{\bn}{\bm_i}}]$-algebra morphism
$\compl{\eps}\colon k[\uZ_{\infty}]\to k(\uZ_{\leq d,\ge\acc{\bn}{\bm_i}})\dblbr{\uZ_{\bullet,<\acc{\bn}{\bm_i}}}$
such that :
\begin{itemize}
\item [(i)] The section ring of the formal neighborhood of
  $\arc(V_{\sigma})$ at $ \eta_{\bn } $ (resp. of   $\arc(W)$ at $\eta'_{\bn}$)
  are both isomorphic to the complete noetherian local ring
  \[k(\uZ_{\leq d,\ge\acc{\bn}{\bm_i}})\dblbr{\uZ_{\bullet,<\acc{\bn}{\bm_i}}}/\ide{\compl{\eps}([\fkj])}.\]
\item [(ii)] For every $i\in \{1,\dots,h\}$, $\compl{\eps}(Z_{i,\acc{\bn}{\bm_i}})$ is invertible.
\item [(iii)] For every $ q\in \{d+1,\dots,h\} $ and $s\in\N$, we have $\compl{\eps}(\Fdisu{q}{\acc{\bn}{\bell_q}+s})=0$.
\end{itemize}
\end{corollary}
\begin{proof}[Proof of corollary \ref{cor:toric:formal:neighbourhood}]
By lemma \ref{lemm:quotient}(ii) and theorem
\ref{theo:main:technical}(iii), it only remains to show that 
$\compl{\eps}([\fkj])=\compl{\eps}([\fki_{\sigma}])$.
 Recall from \S \ref{ssec:jetarc}
the definition of $\HS$. It is enough to show that for every element
$F$ of $\fki_{\sigma}$, one has 
$\compl{\eps}(\HS(F))\in \compl{\eps}([\fkj])\dblbr{\vart}$.
By lemma \ref{lemm:quotient}, there exists a positive integer $N$ such
that $(\prod\limits_{i=1}^dZ_i)^NF\in \fkj$. Thus
\[
\compl{\eps}(\HS((\prod\limits_{i=1}^dZ_i)^N))\,\, \compl{\eps}(\HS(F))\in
\compl{\eps}([\fkj])\dblbr{\vart}.
\]
Since for every $i\in \{1,\dots,d\}$,
$\compl{\eps}(Z_{i,\acc{\bn}{\bm_i}})$ is a unit, 
$\compl{\eps}(\HS((\prod\limits_{i=1}^dZ_i)^N))$ is a regular element
of $k(\uZ_{\leq d,\ge\acc{\bn}{\bm_i}})\dblbr{\uZ_{\bullet,<\acc{\bn}{\bm_i}}}\dblbr{\vart}$,
as well as its projection to
\[k(\uZ_{\leq  d,\ge\acc{\bn}{\bm_i}})\dblbr{\uZ_{\bullet,<\acc{\bn}{\bm_i}}}/\compl{\eps}([\fkj])\dblbr{\vart}.
\]
Since a regular element is not a zero divisor, we infer that 
$\compl{\eps}(\HS(F))\in \compl{\eps}([\fkj])\dblbr{\vart}.$
\end{proof}

\subsection{}\label{ssec:finite:formal:model:rational:arc}
Let us recall the definition of some objects in \cite[Subsection  5.1]{Bou-Seb:toric}, adapted to the notation in the present
section. Denote by $\wt{\eps}\colon k[\uZ_{\infty}]\to k\dblbr{\uZ_{\bullet,<\acc{\bn}{\bm_i}}}$
the unique $k$-algebra morphism mapping, 
for every $i\in \{1,\dots,h\}$, $Z_{i,s}$ to $Z_{i,s}$  for $s<\acc{\bn}{\bm_i}$,
$Z_{i,\acc{\bn}{\bm_i}}$ to $1$ and $Z_{i,s}$ to $0$ for $s>\acc{\bn}{\bm_i}$.

For $ L'\subseteq L $, let $W(\bn,L')$ be
the affine closed $ k $-subscheme of the affine space $ \Spec
(k[\uZ_{\bullet,<\acc{\bn}{\bm_i}}]) $ defined by the ideal
$\ide{\wt{\eps}(F_{\bell,s})\,;\,\bell\in L',s\in\N}$
and $\cW(\bn,L')$ the formal completion of
$W(\bn,L')$ along the origin of $\Spec(k[\uZ_{\bullet,<\acc{\bn}{\bm_i}}])$.
\begin{remark}\label{rema:old:remark}
Let $(\cA,\fkM_{\cA})$ be an object of $\Lacp_k$. Then $\Hom_{\Lacp_k}(\cW(\bn,L'),\cA)$
is in natural bijection with the set of families $\{z_{i,s};\,i\in \{1,\dots,h\},\,0\leq s<\acc{\bn}{\bm_i}\}$
of elements of $\fkM_{\cA}$ such that for every element $\bell\in L'$ one has
\[
F_{\bell}|_{Z_i=\sum_{s=0}^{\acc{\bn}{\bm_i}-1}z_{i,s}\vart^i+\vart^{\acc{\bn}{\bm_i}}}=0.
\]
\end{remark}
The following result follows from \cite[Theorem 5.2]{Bou-Seb:toric}.
\begin{thm}\label{thm:japan}
  For an appropriate choice of $L'\subseteq L$ such that
  $\{\bell_q\,;\,d+1\le q \le h\}\subseteq L'$, for every toric
  valuation $\bn\in N\cap \sigma$ and every arc $\alpha\in \arc(V_{\sigma})^{\circ}_{\bn}(k)$,
the formal neighborhood of $\arc(V_{\sigma})$ at $\alpha$ is
isomorphic to $\cW(\bn,L')\hat\otimes_k k\dblbr{(T_i)_{i\in\N}}$.
\end{thm}
The following lemma shows that for the computation of formal neighborhoods of $k$-rational arcs on $\arc(V_{\sigma})$,
one may also reduce to the complete intersection setting.
\begin{lemma}\label{lem:toric:reduction:complete:intersection:formal:rational}
Let $ \bn \in \sigma\cap N$ be a toric valuation and $ L'$ be a
subset of $L$ such that $ \{\bell_q\,;\,d+1\le q \le h\}\subseteq L'$.
Then $\cW(\bn,L')$ is isomorphic, as a formal $k$-scheme, to $\cW(\bn,\{\bell_q:d+1\le q\le h\})$.
\end{lemma}
\begin{remark}\label{rema:Wn}
Thanks to this lemma, for any $L'\subseteq L$ such that  $\{\bell_q\,;\,d+1\le q \le h\}\subseteq L'$
and any $\bn\in \sigma\cap N$, one may denote $\cW(\bn,L')$ by $\cW(\bn)$.
\end{remark}
\begin{proof}
By remark \ref{rema:old:remark}, there is, for every object
$(\cA,\fkM_{\cA})$ of $\Lacp_k$, a natural inclusion
$\Hom_{\Lacp_k}(\cW(\bn,L'),\cA)\subset \Hom_{\Lacp_k}(\cW(\bn,\{\bell_q\,;\,d+1\le q \le h\}),\cA)$.
To conclude, it suffices to show that this is an equality.
Let $\{z_{i,s};\,i\in \{1,\dots,h\},\,0\leq s<\acc{\bn}{\bm_i}\}$
be a family of elements of $\fkM_{\cA}$ such that, setting
\[
z_i(\vart):=\sum_{s=0}^{\acc{\bn}{\bm_i}-1}z_{i,s}\vart^i+\vart^{\acc{\bn}{\bm_i}},
\]
one has, for every $d+1\le q \le h$,
$
\Fdis{q}|_{Z_i=z_i(\vart)}=0.
$
Let $\bell\in L'$. By lemma \ref{lemm:quotient} there exists a
positive integer $N$ such that
$(\prod_{i=1}^hZ_i)^NF_{\bell}\in \ide{\Fdis{q}\,;\,d+1\le q \le h}$.
Thus 
\[
\left(\prod_{i=1}^hz_i(\vart)\right)^NF_{\bell}|_{Z_i=z_i(\vart)}=0.
\]
Since $z_i(\vart)$ is a Weierstrass polynomial in $\cA\dblbr{t}$, it is a non zero
divisor (see \ref{subsec:WPT}). Thus one infers that $F_{\bell}|_{Z_i=z_i(\vart)}=0$.
That concludes the proof.
\end{proof}

The following proposition performs the aimed comparison in the
complete intersection setting.

\begin{proposition}\label{prop:comparison}
Let $\bn\in\sigma\cap N $ be a toric valuation.  Let
$K:=k(\uZ_{\leq d,\ge\acc{\bn}{\bm_i}})$. Then the residue field of
$\eta'_{\bn }$ is isomorphic to $K$ and the formal neighborhood
of $ \arc(W) $ at the point $ \eta'_{\bn } $ is isomorphic, as a
formal $K$-scheme, to $K\compl{\otimes}_{k}\cW(\bn,\{\bell_q:d+1\le q\le h\})$.
\end{proposition}

\begin{proof}
We still denote by $\wt{\eps}$
the composition of the morphism defined 
 in \S\ref{ssec:finite:formal:model:rational:arc}
with the natural inclusion morphism $k\dblbr{\uZ_{\bullet,<\acc{\bn}{\bm_i}}}\to K\dblbr{\uZ_{\bullet,<\acc{\bn}{\bm_i}}}$.

By corollary \ref{cor:toric:formal:neighbourhood} and the very
definition of $\cW(\bn,\{\bell_q:d+1\le q\le h\})$, it is enough to show
that the quotients of $K\dblbr{\uZ_{\bullet,<\acc{\bn}{\bm_i}}}$ by
the ideals
$\ide{\compl{\eps}([\fkj])}=\ide{\compl{\eps}(\Fdis{q,s}):d+1\leq  q\leq h,s\in \N}$ on one hand,
$\ide{\wt{\eps}([\fkj])}=\ide{\wt{\eps}(\Fdis{q,s}):d+1\leq  q\leq h,s\in \N}$
on the other hand, are isomorphic.

We aim to apply theorem \ref{thm:comparison:technical}.
Set $\Delta:=\{1,\dots,h\}$.
For $i\in \Delta$, set $d_i:=\acc{\bn}{\bm_i}$;
for $0\leq s< \acc{\bn}{\bm_i}$ set $X_{i,s}:=Z_{i,s}$
and for $s\geq \acc{\bn}{\bm_i}$ set
$x_{i,s}:=\compl{\eps}(Z_{i,s})$.
For $i\in \Delta$, set
\[
\cY_{i}(\vart):=\sum_{s\in \N}\compl{\eps}(Z_{i,s})\vart^s
=\sum_{s=0}^{\acc{\bn}{\bm_i}-1}X_{i,s}\vart^s+\sum_{s\geq \acc{\bn}{\bm_i}}x_{i,s}\vart^s
\]
\[
\text{and}\quad \wt{\cY}_{i}(\vart):=\sum_{s=0}^{\acc{\bn}{\bm_i}-1}X_{i,s}\vart^s+\vart^{\acc{\bn}{\bm_i}}.
\]
For every $P\in k[\uZ]$, we then have (see
notation \ref{not:homogeneous:coefficients} and remark
\ref{rema:definition:homogeneous:coefficients})
$P_{s,\cuY(\vart)}=\compl{\eps}(P_{s})$
and $P_{s,\wt{\cuY}(\vart)}=\wt{\eps}(P_{s})$.

Set $\Omega:=\{d+1,\dots,h\}$ and for $q\in\Omega$ set $P_q:=\Fdis{q}$.

Assumption (a) is a consequence of
corollary \ref{cor:toric:formal:neighbourhood}.

With our identifications, the nonzero integer $ c_q $ defined in
the statement of theorem \ref{thm:comparison:technical} is
\[
c_q=\sum\limits_{i=1}^{h} (\bell^+_{q})_i\acc{\bn}{\bm_i}
=
\acc{\bn}{\sum\limits_{\substack{i=1\\\ell_{q,i}\ge 0}}^{h}\ell_{q,i}\bm_i}
=\acc{\bn}{\bell_q}. 
\]
Then still by corollary \ref{cor:toric:formal:neighbourhood},
for every $s\in \N$ we have $\compl{\eps}(\Fdisu{q}{\acc{\bn}{\bell_q}+s})=0$.
Thus $\Fdisu{q}{\acc{\bn}{\bell_q}+s,\cuY(\vart)}=0$ and assumption $ (b) $ holds.
That concludes the proof.
\end{proof}

\subsection{}\label{ssec:comparison}
Now one can state the main theorem of the article.
It illustrates the striking fact that not only the isomorphism class of the formal neighborhood of a
generic $k$-rational arc of the Nash set associated with a toric valuation
is constant (as observed in \cite{Bou-Seb:toric}) but moreover the
involved isomorphism class is encoded in some sense in the formal
neighborhood of the generic point of the Nash set. This could be
interpreted as the fact that the arc scheme of a toric variety is
analytically a product along the Nash set associated with the toric valuation.
\begin{thm}
\label{thm:comparison}
Let $ \bn\in \sigma\cap N $ be a toric valuation. Let $\eta_{\bn}$ be
the generic point of the Nash set $\mds_{\bn}$.
Let $\cW(\bn)$ be the noetherian formal $k$-scheme defined in remark \ref{rema:Wn}.

Then there exists a nonempty open subset $U_{\bn}$ of the Nash set $\mds_{\bn}$ such that:
\begin{itemize}\renewcommand{\labelitemi}{$\bullet$}
\item [(i)] The formal neighborhood of $\arc(V_{\sigma})$ at $\eta_{\bn }$ is isomorphic, as a formal
$\kappa(\eta_{\bn})$-scheme, to $\kappa(\eta_{\bn})\widehat{\otimes}_{k}\cW(\bn)$.
In particular it is isomorphic to the formal spectrum of the
completion of an essentially of finite type local $\kappa(\eta_{\bn})$-algebra.
\item [(ii)] For any arc $\alpha\in U_{\bn}(k)$,
  the formal neighborhood of $\arc(V_{\sigma})$ at $\alpha$ is
  isomorphic to $\cW(\bn)\widehat{\otimes}_k k\dblbr{(T_i)_{i\in\N}}$.
\end{itemize}
\end{thm}
\begin{proof}
One takes $U_{\bn}:=\arc(V_{\sigma})^{\circ}_{\bn}$ and 
one combines  proposition \ref{prop:comparison},
theorem \ref{thm:japan}, lemma \ref{lem:toric:reduction:complete:intersection:formal:rational},
and corollary \ref{cor:toric:formal:neighbourhood}(i).
\end{proof}

\subsection{}\label{ssec:decomposition}
 An element $\bn\in N\setminus \{0\}$ is said to be {\em primitive} if it can not be
 written as $d\bn'$ where $\bn'\in N$ and $d$ is an integer $>1$.
  An element $\bn\in N\cap \sigma\setminus \{0\}$ is said to be
  {\em indecomposable} if it can not be written $\bn=\bn_1+\bn_2$ with
  $\bn_1,\bn_2\in N\cap \sigma\setminus \{0\}$. A decomposition of
  $\bn$ into indecomposable elements is a decomposition
  $\bn=\sum_{i=1}^r\bn_i$ where $r$ is a positive integer and the
  $\bn_i$'s are indecomposable elements in $N\cap \sigma\setminus \{0\}$; the {\em length} of such a
  decomposition is $r$.

\subsection{}\label{subsec:cor}  Using results of our previous work \cite{Bou-Seb:toric}, we
deduce, as a straightforward by-product of theorem \ref{thm:comparison}, the
following corollary. The result has been obtained independently by
Reguera using a different approach (see \cite{Reg:tor}).
\begin{corollary}\label{coro:dim:irrcomp}
Let $ \bn\in \sigma\cap N $ be a toric valuation of $V_{\sigma}$ and $\eta_{\bn}$ be
the generic point of the Nash set $\mds_{\bn}$.

Then there is a natural bijection between the set of irreducible components
of the formal neighborhood $\compl{\arc(V_{\sigma})_{\eta_{\bn}}}$ and the
set of decompositions of $\bn$ into a sum of indecomposable elements of the
semigroup $N\cap \sigma$.
The dimension of the component corresponding to a given decomposition of $\bn$ is the length of the decomposition.
In particular the dimension of $\compl{\arc(V_{\sigma})_{\eta_{\bn}}}$
is equal to the maximal length of such a decomposition of $\bn$.
\end{corollary}
\begin{proof}
The fact that the conclusion holds for $\cW(\bn)$ 
is shown in section (3) of the proof of \cite[Theorem
6.3]{Bou-Seb:toric}. In the latter $\cW(\bn)$ is denoted by $\tilde{\mathscr{Y}}_n$.
Though it is assumed in the statement of the theorem that $\bn$ is
primitive, this is not used in the aforementioned section of the proof.
The corollary then follows from theorem \ref{thm:comparison}.
\end{proof}
\subsection{}
It remains to explain why Theorem \ref{theo:strong:essential} stated in the introduction
is a consequence of corollary \ref{coro:dim:irrcomp}.
Let $\bn\in N\cap \sigma$ be a primitive integral point
representing a toric valuation of multiplicity $1$ and assume that
$v$ is  centered in the singular locus of $V_{\sigma}$.
Then, by \cite[Theorem 1.10]{BouGon} and \cite[Theorem 1.2]{Bou:divess},
$\bn$ is indecomposable if and only if $v$ is a strongly essential
valuation, in the sense given in the introduction.
Thus theorem \ref{theo:strong:essential} is indeed a consequence of
corollary \ref{coro:dim:irrcomp} (using again the proof of \cite[Theorem 6.3]{Bou-Seb:toric}).

\subsection{}\label{subsec:exam} We end this section with an explicit
example of computation of the formal neighborhood of the generic point
of the Nash set associated with a toric valuation. See
\cite{MMC:PhD} for more details.

Let $ N=M=\Z^{2} $, $ \sigma $ be the cone of $ \R^2 $ generated by $ (1,0) $ and $ (1,2)$,
and  $ V_{\sigma} $  be the associated affine toric variety. 
The semigroup $ S_{\sigma} $ is minimally generated by $ m_1=(0,1) $,
$ m_2=(1,0) $ and $ m_3=(2,-1)$. We observe that $ m_1 $ and $ m_2 $
form a $ \Z $-basis of $ M $ and the relation $ m_1+m_3=2m_2 $ generates all the relations between elements of $ S_{\sigma} $.
Thus, setting $F:=Z_1Z_3-Z_2^2$,  the ideal of $ V_{\sigma} $ in $k[Z_1,Z_2,Z_3]$
is the ideal generated by $F$. The ideal of $\arc(V_{\sigma})$ in the ring $ k[\uZ_{\infty}]=k[Z_{1,s},Z_{2,s},Z_{3,s};s\in \N]$
is generated by $\{F_s;s\in \N\}$ where $F_{s}=\sum\limits_{r=0}^{s}(Z_{1,s-r}Z_{3,r}-Z_{2,s-r}Z_{2,r})$.

We now consider the toric valuation $\ord_{\bn}$ of $V_{\sigma}$
corresponding to $\bn=(1,1)\in \sigma\cap N$. The prime ideal of
$k[\uZ_{\infty}]$ corresponding to the generic point $\eta_{\bn}$ of the
Nash set associated with $\ord_{\bn}$ is the radical of the ideal $\ide{Z_{1,0}, Z_{2,0}, Z_{3,0}}$.
The residue field of $\eta_{\bn}$ is isomorphic to $K:=k(Z_{1,s},Z_{2,s};s\ge 1)$.

Denote by $\{z_{3,s}; s\geq 1\}$ the unique family of elements of $K$ such that
for every $s\geq 2$, one has
\[
\sum\limits_{r=1}^{s-1}\left(Z_{1,s-r}\cdot z_{3,r}-Z_{2,s-r}Z_{2,r}\right)=0.
\]
Note that the latter is a triangular invertible $K$-linear system in the $z_{3,s}$'s.

Now let $\{\cZ_{3,r};r\geq 1\}$ be the unique family of elements of $K\dblbr{Z_{1,0}, Z_{2,0}, Z_{3,0}}$
such that
\begin{enumerate}
\item for every $s\geq 1$, one has $\cZ_{3,s}=z_{3,s}\pmod{\ide{Z_{1,0}, Z_{2,0}, Z_{3,0}}}$;
\item for every $s\geq 2$, one has
  \[
  Z_{1,s}Z_{3,0}-Z_{2,0}Z_{2,s}+\sum\limits_{r=1}^{s}\left(Z_{1,s-r}\cdot\cZ_{3,r}-Z_{2,s-r}Z_{2,r}\right)=0.
  \]
\end{enumerate}
Explicit truncations of the series $\cZ_{3,s}$ may be obtained
by applying effectively the Hensel's lemma, in other words by successive
approximations, though explicit computations quickly become cumbersome. For example one has
\[
\cZ_{3,1}=\frac{Z_{2,1}^2}{Z_{1,1}}+\frac{Z_{1,0}Z_{1,2}Z_{2,1}^2}{Z_{1,1}^3}-\frac{2Z_{1,0}Z_{2,1}Z_{2,2}}{Z_{1,1}^2}-\frac{Z_{1,2}Z_{3,0}}{Z_{1,1}}+\frac{2Z_{2,0}Z_{2,2}}{Z_{1,1}}
\pmod{\ide{Z_{1,0}, Z_{2,0}, Z_{3,0}}^2}
\]
and
\begin{align*}
\cZ_{3,2}=&-\frac{Z_{1,2}Z_{2,1}^2}{Z_{1,1}^2}+\frac{2Z_{2,1}Z_{2,2}}{Z_{1,1}}-\frac{2Z_{1,0}Z_{1,2}^2Z_{2,1}^2}{Z_{1,1}^4}+\frac{4Z_{1,0}Z_{1,2}Z_{2,1}Z_{2,2}}{Z_{1,1}^3}\\&+\frac{Z_{1,0}Z_{1,3}Z_{2,1}^2}{Z_{1,1}^3}-\frac{Z_{1,0}Z_{2,2}^2}{Z_{1,1}^2}-\frac{2Z_{1,0}Z_{2,1}Z_{2,3}}{Z_{1,1}^2}+\frac{Z_{1,2}^2Z_{3,0}}{Z_{1,1}^2}-\frac{2Z_{1,2}Z_{2,0}Z_{2,2}}{Z_{1,1}^2}\\&-\frac{Z_{1,3}Z_{3,0}}{Z_{1,1}}+\frac{2Z_{2,0}Z_{2,3}}{Z_{1,1}}\pmod{\ide{Z_{1,0},
    Z_{2,0}, Z_{3,0}}^2}.
\end{align*}
Then the formal neighborhood of $\eta_{\bn}$ in $\arc(V_{\sigma})$ is isomorphic to the
formal spectrum of
\[
K\dblbr{Z_{1,0}, Z_{2,0}, Z_{3,0}}/\ide{Z_{1,0}Z_{3,0}-Z_{2,0}^2,Z_{1,1}Z_{3,0}+Z_{1,0}\cZ_{3,1}-2Z_{2,0}Z_{2,1}}.
\]
Note that it is not clear that the latter is the completion of an
essentially of finite type local $K$-algebra.

\subsection{} Using our comparison theorem, the computation of the
formal neighborhood of $\eta_{\bn}$ in $\arc(V_{\sigma})$
may also be done in the following much more straightforward way. First we compute the formal scheme
$\cW(\bn)$ defined in \ref{ssec:finite:formal:model:rational:arc}.
We have the following equality in $k[Z_{1,0},Z_{2,0},Z_{3,0},t]$:
\[
F|_{Z_j=t+Z_{j,0}}=(t+Z_{1,0})(t+Z_{3,0})-(t+Z_{2,0})^2=(Z_{1,0}+Z_{3,0}-2Z_{2,0})t+Z_{1,0}Z_{3,0}-Z_{2,0}^2.
\]
We deduce that 
\[
\cW(\bn)=\Spf\left(\dfrac{k\dblbr{Z_{1,0},Z_{2,0},Z_{3,0}}}{\ide{Z_{1,0}Z_{3,0}-Z_{2,0}^2,Z_{1,0}+Z_{3,0}-2Z_{2,0}}}\right)
\]
and that the formal neighborhood of $\eta_{\bn}$ in $\arc(V_{\sigma})$ is
isomorphic to 
\[
\Spf\left(\dfrac{K\dblbr{Z_{1,0},Z_{2,0},Z_{3,0}}}{\ide{Z_{1,0}Z_{3,0}-Z_{2,0}^2,Z_{1,0}+Z_{3,0}-2Z_{2,0}}}\right).
\]
In addition, it is not difficult to see that $\cW(\bn)$ is isomorphic
to $\Spf(k\dblbr{Z_{1,0},Z_{2,0} }/\ide{Z_{1,0}^2})$.

\bibliographystyle{amsplain}
\bibliography{Toric}

\providecommand{\bysame}{\leavevmode\hbox to3em{\hrulefill}\thinspace}
\providecommand{\MR}{\relax\ifhmode\unskip\space\fi MR }
\providecommand{\MRhref}[2]{%
  \href{http://www.ams.org/mathscinet-getitem?mr=#1}{#2}
}
\providecommand{\href}[2]{#2}
\begin{thebibliography}{10}

\bibitem{BNS-Proceedings}
David Bourqui, Johannes Nicaise, and Julien Sebag, \emph{Arc schemes in
  geometry and differential algebra}, Arc Schemes and Singularities, World Sci.
  Publ., 2020.

\bibitem{BouSeb:confluentes}
David Bourqui and Julien Sebag, \emph{The {D}rinfeld-{G}rinberg-{K}azhdan
  theorem for formal schemes and singularity theory}, Confluentes Math.
  \textbf{9} (2017), no.~1, 29--64.

\bibitem{BouSeb:false}
\bysame, \emph{The {D}rinfeld-{G}rinberg-{K}azhdan theorem is false for
  singular arcs}, J. Inst. Math. Jussieu \textbf{16} (2017), no.~4, 879--885.
  \MR{3680347}

\bibitem{BS-IJM}
\bysame, \emph{The minimal formal models of curve singularities}, Internat. J.
  Math. \textbf{28} (2017), no.~11, 1750081, 23. \MR{3714357}

\bibitem{BouSeb:smooth:arcs}
\bysame, \emph{Smooth arcs on algebraic varieties}, J. Singul. \textbf{16}
  (2017), 130--140. \MR{3670512}

\bibitem{Bou-Seb:toric}
\bysame, \emph{Finite formal model of toric singularities}, J. Math. Soc. Japan
  \textbf{71} (2019), no.~3, 805--829. \MR{3984243}

\bibitem{Bou-Seb:Nash}
\bysame, \emph{The local structure of arc schemes}, Arc Schemes and
  Singularities, World Sci. Publ., 2020.

\bibitem{BouCohArc}
A.~Bouthier, \emph{Cohomologie etale des espaces d'arc}, Preprint
  {\verb|https://arxiv.org/abs/1509.02203v6|}, 2020.

\bibitem{BouNgoSak}
A.~Bouthier, B.~C. Ng{\^o}, and Y.~Sakellaridis, \emph{On the formal arc space
  of a reductive monoid}, Amer. J. Math. \textbf{138} (2016), no.~1, 81--108.
  \MR{3462881}

\bibitem{Bou:divess}
Catherine Bouvier, \emph{Diviseurs essentiels, composantes essentielles des
  vari\'et\'es toriques singuli\`eres}, Duke Math. J. \textbf{91} (1998),
  no.~3, 609--620. \MR{1604179}

\bibitem{BouGon}
Catherine Bouvier and G\'erard Gonzalez-Sprinberg, \emph{Syst\`eme
  g\'en\'erateur minimal, diviseurs essentiels et {$G$}-d\'esingularisations de
  vari\'et\'es toriques}, Tohoku Math. J. (2) \textbf{47} (1995), no.~1,
  125--149. \MR{1311446}

\bibitem{CNS:book}
Antoine Chambert-Loir, Johannes Nicaise, and Julien Sebag, \emph{Motivic
  integration}, Progress in Mathematics, vol. 325, Birkh\"{a}user, 2018.

\bibitem{ChDeFDoc:embcodim}
Christopher Chiu, Tommaso de~Fernex, and Roi Docampo, \emph{Embedding
  codimension of the space of arcs}, \url{https://arxiv.org/abs/2001.08377},
  2020.

\bibitem{ChHa:degen}
Christopher Chiu and Herwig Hauser, \emph{On the formal neighborhood of
  degenerate arcs}, available at
  \url{https://homepage.univie.ac.at/herwig.hauser/Publications/}, 2016.

\bibitem{CLS}
David~A. Cox, John~B. Little, and Henry~K. Schenck, \emph{Toric varieties},
  Graduate Studies in Mathematics, vol. 124, American Mathematical Society,
  Providence, RI, 2011. \MR{2810322}

\bibitem{deFDoc:terminal}
Tommaso de~Fernex and Roi Docampo, \emph{Terminal valuations and the {N}ash
  problem}, Invent. Math. \textbf{203} (2016), no.~1, 303--331. \MR{3437873}

\bibitem{deFernex-Docampo}
\bysame, \emph{Differentials on the arc space}, Duke Math. J. \textbf{169}
  (2020), no.~2, 353--396. \MR{4057146}

\bibitem{DGKv1}
Vladimir Drinfeld, \emph{On the {G}rinberg–{K}azhdan formal arc theorem},
  2002.

\bibitem{ELL04}
Lawrence Ein, Robert Lazarsfeld, and Mircea Musta\c{t}\v{a}, \emph{Contact loci
  in arc spaces}, Compos. Math. \textbf{140} (2004), no.~5, 1229--1244.
  \MR{2081163}

\bibitem{MR2979864}
Javier Fern\'{a}ndez~de Bobadilla and Mar\'{\i}a~Pe Pereira, \emph{The {N}ash
  problem for surfaces}, Ann. of Math. (2) \textbf{176} (2012), no.~3,
  2003--2029. \MR{2979864}

\bibitem{GK}
M.~Grinberg and D.~Kazhdan, \emph{Versal deformations of formal arcs}, Geom.
  Funct. Anal. \textbf{10} (2000), no.~3, 543--555.

\bibitem{Ish:arcs:valuations}
Shihoko Ishii, \emph{Arcs, valuations and the {N}ash map}, J. Reine Angew.
  Math. \textbf{588} (2005), 71--92. \MR{2196729}

\bibitem{Ish:maximal}
\bysame, \emph{Maximal divisorial sets in arc spaces}, Algebraic geometry in
  {E}ast {A}sia---{H}anoi 2005, Adv. Stud. Pure Math., vol.~50, Math. Soc.
  Japan, Tokyo, 2008, pp.~237--249. \MR{2409559}

\bibitem{Lang}
Serge Lang, \emph{Algebra}, third ed., Graduate Texts in Mathematics, vol. 211,
  Springer-Verlag, New York, 2002. \MR{1878556 (2003e:00003)}

\bibitem{Lej:arcsan}
M.~Lejeune-Jalabert, \emph{Arcs analytiques et r\'esolution minimale des
  singularit\'es des surfaces quasi homog\`enes}, S\'eminaire sur les
  singularit\'es des surfaces (1976-1977) (fr), talk:18.

\bibitem{MR2905305}
Monique Lejeune-Jalabert and Ana~J. Reguera, \emph{Exceptional divisors that
  are not uniruled belong to the image of the {N}ash map}, J. Inst. Math.
  Jussieu \textbf{11} (2012), no.~2, 273--287. \MR{2905305}

\bibitem{MMC:PhD}
Mario Morán~Cañón, \emph{Study of the scheme structure of arc scheme},
  Th{\`e}se de doctorat, Universit{\'e} de Rennes 1, 2020.

\bibitem{MouReg}
Hussein Mourtada and Ana~J. Reguera, \emph{Mather discrepancy as an embedding
  dimension in the space of arcs}, Publ. Res. Inst. Math. Sci. \textbf{54}
  (2018), no.~1, 105--139. \MR{3749346}

\bibitem{Nash}
John~F. Nash, Jr., \emph{Arc structure of singularities}, Duke Math. J.
  \textbf{81} (1995), no.~1, 31--38 (1996), A celebration of John F. Nash, Jr.

\bibitem{Ngo:Wei}
Bao~Chau Ng{\^o}, \emph{Weierstrass preparation theorem and singularities in
  the space of non-degenerate arcs}, Preprint
  {\verb|https://arxiv.org/abs/1706.05926|}, 2017.

\bibitem{Reg:tor}
Ana~J. Reguera, \emph{Arc spaces and wedge spaces for toric varieties},
  Preprint {\verb|https://hal.archives-ouvertes.fr/hal-03031886/document|}.

\bibitem{Reg:CSL}
\bysame, \emph{A curve selection lemma in spaces of arcs and the image of the
  {N}ash map}, Compos. Math. \textbf{142} (2006), no.~1, 119--130. \MR{2197405
  (2008a:14007)}

\bibitem{Reguera}
\bysame, \emph{Towards the singular locus of the space of arcs}, Amer. J. Math.
  \textbf{131} (2009), no.~2, 313--350.

\bibitem{Reg:coo}
\bysame, \emph{Coordinates at stable points of the space of arcs}, J. Algebra
  \textbf{494} (2018), 40--76. \MR{3723170}

\bibitem{Stu96}
Bernd Sturmfels, \emph{Gr\"{o}bner bases and convex polytopes}, University
  Lecture Series, vol.~8, American Mathematical Society, Providence, RI, 1996.
  \MR{1363949}

\end{thebibliography}

\end{document}